\documentclass[11pt,a4paper]{amsart}
\usepackage{amsmath,amssymb,amsthm}
\usepackage{ascmac}
\usepackage{showkeys}
\usepackage{color}

\usepackage[utf8]{inputenc}
\theoremstyle{definition}
 \newtheorem{dfn}{Definition}
 \newtheorem{remark}[dfn]{Remark}

\theoremstyle{plain}
 \newtheorem{thm}[dfn]{Theorem}
 
 \newtheorem{lem}[dfn]{Lemma}
 \newtheorem{cor}[dfn]{Corollary}

\numberwithin{equation}{section}

\newcommand{\bu}{{\bold u}}

\newcommand{\bv}{{\bold v}}
\newcommand{\bw}{{\bold w}}
\newcommand{\ba}{{\bold a}}

\newcommand{\bff}{{\bold f}}

\newcommand{\bG}{{\bold G}}

\newcommand{\bI}{{\bold I}}

\newcommand{\bS}{{\bold S}}
\newcommand{\bT}{{\bold T}}
\newcommand{\bU}{{\bold U}}

\newcommand{\dv}{{\rm div}\,}

\newcommand{\BA}{{\Bbb A}}

\newcommand{\BR}{{\Bbb R}}
\newcommand{\BC}{{\Bbb C}}

\newcommand{\BN}{{\Bbb N}}
\newcommand{\BG}{{\Bbb G}}

\newcommand{\BI}{{\Bbb I}}

\newcommand{\BF}{{\Bbb F}}

\newcommand{\BZ}{{\Bbb Z}}

\newcommand{\CA}{{\mathcal A}}

\newcommand{\CD}{{\mathcal D}}
\newcommand{\CE}{{\mathcal E}}

\newcommand{\CL}{{\mathcal L}}

\newcommand{\CR}{{\mathcal R}}
\newcommand{\CS}{{\mathcal S}}

\newcommand{\CH}{{\mathcal H}}

\newcommand{\bg}{{\bold g}}
\newcommand{\bh}{{\bold h}}
\newcommand{\pd}{\partial}

\newcommand{\HS}{\BR^N_+}
\newcommand{\WS}{\BR^N}

\renewcommand{\d}{\,\mathrm{d}}

\begin{document}
\title[]{$L_1$ approach to the compressible viscous fluid flows 
in the half-space}

\author[]{Jou chun Kuo}
\address{(J. Kuo) {School of Science and Engineering},
Waseda University \\
3-4-5 Ohkubo Shinjuku-ku, Tokyo, 169-8555, Japan }		
\email{kuojouchun@asagi.waseda.jp} 
\author[]{Yoshihiro Shibata}
\address{(Y. Shibata) 
{ Emeritus of Waseda University \\
Adjunct faculty member in the Department of Mechanical Engineering 
and	Materials Science, University of Pittsburgh, USA
}}
\email{yshibata325@gmail.com}


\thanks{The second author was partially supporte by 
and the second author was partially supported by JSPS KAKENHI 
Grant Number 22H01134  and Top Global University Project}


\date{\today}   

\maketitle


\begin{abstract}
In this paper, we prove the local well-posedness for the Navier-Stokes equations describing the motion of
isotropic barotoropic compressible viscous fluid flow
in the half-space $\BR^N_+ = \{x = (x_1,\ldots, x_d) \in \BR^N \mid x_d>0\}$ with 
non-slip boundar condition, where $\HS$ is  the 
fluid domain.   
The density part of our solutions belongs to  $W^1_1((0, T), B^s_{q,1}(\HS))
\cap L_1((0, T), B^{s+1}_{q,1}(\HS))$ and the velocity part  of our solutions  
$W^1_1((0, T), B^{s}_{q,1}(\HS)^N) \cap L_1((0, T), B^{s+2}_{q,1}(\HS)^N)$, 
where $B^\mu_{q,1}(\HS)$ denotes the inhomogeneous Besov
space on $\HS$. Namely, we solve the equations in  
the $L_1$ in time and $B^{s+1}_{q,1}(\HS) \times B^s_{q,1}(\HS)^N$ in space maximal regularity
framework. 
We use Lagrange transformation to eliminate the convection term $\bv\cdot\nabla\rho$ and 
we use an analytic semigroup approach.  
We only assume  the strictly positiveness of initial mass density. 
An essential assumption is that  $-1+N/q \leq  s < 1/q$ if $N-1 < q  < 2N$ and 
$-N/q < s < 1/q$ if $q \geq 2N$, where $N/q$ is the crucial 
order to obtain  $\|\nabla \bu\|_{L_\infty} \leq C\|\nabla\bu\|_{B^{N/q}_{q,1}}$. 
\end{abstract}

\section{Introduction}

Let $1 < q < \infty$ and $-1+N/q \leq s < 1/q$, where $N$ is the space dimension. 
In this paper, we use the  $L_1$--$B^{s+1}_{q,1}\times B^s_{q,1}$ maximal regularity framework to show
 the local well-posedness of the Navier-Stokes
equations describing the isotropic motion of the compressible viscous fluid flows
in the half-space. Let
$$\HS=\{x=(x_1, \ldots, x_N) \in \BR^N \mid x_N > 0\}, 
\quad
\pd\HS= \{x=(x_1, \ldots, x_N) \in \BR^N \mid x_N = 0\}.
$$
The equations considered in this paper read as
\begin{equation}\label{ns:1}\left\{\begin{aligned}
\rho_t + \dv(\rho\bv) = 0&&\quad&\text{in $\HS\times(0, T)$}, \\
\rho(\bv_t + \bv\cdot\nabla\bv) - \alpha\Delta \bv - \beta\nabla\dv\bv
+ \nabla P(\rho) = 0&&\quad&\text{in $\HS\times(0, T)$}, \\
\bv|_{\pd\HS} = 0, \quad(\rho, \bv) =(\rho_0, \bv_0)
&&\quad&\text{in $\HS$}.
\end{aligned}\right.\end{equation}
Here, 
$\alpha$ and $\beta$ denote respective
the viscosity coefficients and the second viscosity 
coefficients satisfying the conditions
\begin{equation}\label{assump:1.1}
\alpha > 0, \quad \alpha+\beta >0,
\end{equation}
and 
$P(\rho)$ is a smooth function defined on $(0, \infty)$ satisfying $P'(\rho) > 0$, 
that is,  the barotropic fluid is considered.

\par
The main result of this paper is the following theorem. 
\begin{thm}\label{thm:1} Let $1 < q < \infty$ and $-1+1/q < s < 1/q$.  Moreover,
we assume that 
\begin{equation}\label{assump:s}\left\{\begin{aligned}
-1 + \frac{N}{q} \leq &\,s < \frac1q &\quad &\text{ for $N-1 < q < 2N$}, \\
-\frac{N}{q} < &\,s <\frac1q &\quad&\text{ for $q \geq 2N$}.
\end{aligned}\right.\end{equation}
Let $\eta_0$ be a  function satisfying the following two conditions:
\begin{enumerate}
\item For some positive constants $\rho_1$ and $\rho_2$, it holds 
\begin{equation}\label{assump:0}
\rho_1 < \eta_0(x) < \rho_2, \quad  \rho_1 < P'(\eta_0(x)) < \rho_2
\quad(x \in \overline{\HS}).
\end{equation}
\item 
There exist a positive constant $\rho_*$ and a function   $\tilde\eta_0
\in B^{s+1}_{q,1}(\HS)$ such that $\eta_0 = \gamma_* + \tilde\eta_0$.
\end{enumerate}
Here and in the sequel, $\rho_*$ is a positive constant describing the mass density of the 
reference domain $\HS$, and $B^\mu_{q,p}$ denotes the standard Besov space. \par 
Then, 
there exist  small numbers $T>0$ and $\sigma>0$  such that 
for any initial data 
$\rho_0= \gamma_*+\tilde\rho_0$ with $\tilde\rho_0
 \in B^{s+1}_{q,1}(\HS)$ and $\bv_0 \in B^s_{q,1}(\HS)$, 
problem \eqref{ns:1} admits unique solutions $\rho$ and $\bv$
satisfying the regularity conditions:
\begin{equation}\label{main.reg} \begin{aligned}
\rho-\rho_0 &\in L_1((0, T), B^{s+1}_{q,1}(\HS)) \cap W^1_1((0, T), B^s_{q,1}(\HS)),
\\
\bv &\in L_1((0, T), B^{s+2}_{q,1}(\HS)^N) \cap W^1_1((0, T), B^s_{q,1}(\HS)^N)
\end{aligned}\end{equation}
provided that $\|\tilde\rho_0-\tilde\eta_0\|_{B^{s+1}_{q,1}(\HS)} 
\leq \sigma$
and $\bv_0$
satisfies the compatibility condition: $\bv_0|_{\pd\HS}=0$.
\end{thm}
\begin{remark}\label{rem:1}
If $q \geq 2N$, then we see that $-1+ N/q \leq -N/q$.  Thus, $-1+N/q 
\leq s$. And, what $N-1 < q$ is necessary to have the situation: $-1+ N/q < 1/q$. 
\end{remark}
R. Danchin and R. Tolksdorf \cite{DT22} proved the local and global well-posedness 
of equations \eqref{ns:1} in the $L_1$ in time and $B^{N/q}_{q,1} \times B^{N/q -1}_{q,1}$ in space 
maximal regularity framework for some $q \in (2, min(4, 2N/(N-2))$, 
and the main assumption is that the fluid domain is bounded.  
To obtain  the $L_1$ in time  maximal regularity of solutions to the linearized equations,
so called Stokes equations in the compressible fluid flow case, in \cite{DT22} 
they used their extended version of 
Da Prato and Grisvard theory \cite{DG}, which was a first result concerning $L_1$ maximal regularity
for  continuous analytic semigroups.  In \cite{DT22} , they assumed that the fluid domain is bounded,
which  seems to be necessary to obtain the linear theory for Lam\'e equations cf. \cite[Sect. 3]{DT22}
in their argument  \par
The final goal of our study is 
to solve equations \eqref{ns:1}  if the fluid domain is a general $C^2$ class domain. 
If the fluid domain is the whole space, a number of results have been estabilished 
\cite{CD10, D08, AP07, H11} and references  given therein.   
Thus, our interest is in the initial boundary value problem case.
As a first step of our study, in this paper we consider equations in the half-space, 
namely the model problem for the initial boundary value problem. 
To explain our approach,  let $A$ be Stokes operator, $\bI$ the identity operator, 
and $\Sigma_\mu$ a parabolic sector defined by 
\begin{equation}\label{resol:0}
\Sigma_\mu = \{\lambda \in \BC\setminus\{0\} \mid |\arg\lambda| \leq \pi-\mu\}.
\end{equation}
for $\mu \in (0, \pi/2)$. Let $X = B^{s+1}_{q,1}\times B^s_{q,1}$ be a underlying space of 
the operator $A$, and we may prove that there exists a large number $\gamma > 0$ such that 
the resolvent $(\lambda\bI + A)^{-1}$ exists as a surjective operator on $X$
for $\lambda \in \Sigma_\mu + \gamma$. 
Notice that $B^{s+1}_{q,1}$ is the underlying space of the mass density and 
$B^s_{q,1}$ the underlying space of the velocity field. 
 Thus, $A$ generates a $C_0$ analytic semigroup $\{e^{-At}\}_{t\geq 0}$.
  But, to prove that this is $L_{1, {\rm loc}}$ in tim, we prove that $(\lambda\bI + A)^{-1}$ is 
divided as $(\lambda\bI + A)^{-1} = \CA_1(\lambda) + \CA_2(\lambda)$, where $\CA_i(\lambda)$
($i=1,2$) satisfy the following estimates:
\begin{align}&\label{spectral:0}\left\{\begin{aligned}
\|\lambda\CA_1(\lambda)\|_{\CL(X, X_\sigma)} &\leq C|\lambda|^{-\sigma/2}, \\
\|\lambda\pd_\lambda\CA_1(\lambda)\|_{\CL(X, X_{-\sigma})} & \leq C|\lambda|^{-(1-\sigma/2)};
\end{aligned}\right. \\
&\label{spectral:1}\left\{\begin{aligned}
\|\lambda\CA_2(\lambda)\|_{\CL(X, X)} &\leq C|\lambda|^{-\sigma/2}, \\
\|\lambda\pd_\lambda\CA_2(\lambda)\|_{\CL(X, X)} & \leq C|\lambda|^{-(1-\sigma/2)}
\end{aligned}\right. \end{align}
for every $\lambda \in \Sigma_\mu + \gamma$.
Here,  $\CL(E, F)$ denotes the set of all bounded linear operators from $E$ into $F$, 
 $X_{\pm\sigma} = B^{s+1\pm\sigma}_{q,1}\times B^{s\pm\sigma}_{q,1}$, and 
 $\sigma$ is a very small positive number such that $ -1/q < s-\sigma < s < s+\sigma < 1/q$.
As is known in theory of continuous analytic semigroup \cite{YK}, $e^{-At}$ can be represented as
$$e^{-At}F = \frac{1}{2\pi i}\int_{\Gamma+\gamma} (\lambda\bI + A)^{-1}F\,\d\lambda$$
where $\Gamma = \Gamma_+ \cup \Gamma_-$ and $\Gamma_\pm = \{\lambda \in \BC \mid
\lambda = re^{\pm i(\pi-\mu)}, r \in [0, \infty)\}$. 
Let $\CE_1(t)$ and $\CE_2(t)$ be semigroups generated defined by
$$\CE_i(t)F = \frac{1}{2\pi i}\int_{\Gamma+\gamma} \CA_i(\lambda)F\,\d\lambda.$$
Obvisously,   $e^{-At} = \CE_1(t) + \CE_2(t)$.  Moreover, using \eqref{spectral:0}, we see that 
\begin{align*}
\|\pd_t \CE_1(t)F\|_X &\leq Ce^{\gamma t}t^{-1+\frac{\sigma}{2}}\|F\|_{X_\sigma}, \\
\|\pd_t \CE_1(t)F\|_X &\leq Ce^{\gamma t}t^{-1-\frac{\sigma}{2}}\|F\|_{X_{-\sigma}}.
\end{align*}
Thus, using real interpolation method, we have
$$\int^\infty_0 e^{-\gamma t}\|\pd_t \CE_1(t)F\|_X\,dt \leq C\|F\|_{(X_\sigma, X_{-\sigma})_{1/2, 1}}.
$$
And also, using  \eqref{spectral:1}, we have
\begin{align*}
\|\pd_t \CE_2(t)F\|_X &\leq Ce^{\gamma t}t^{-1+\frac{\sigma}{2}}\|F\|_X,  \\
\|\pd_t \CE_2(t)F\|_X &\leq Ce^{\gamma t}t^{-1-\frac{\sigma}{2}}\|F\|_X.
\end{align*}
Thus, using real interpolation method, we have
$$\int^\infty_0 e^{-\gamma t}\|\pd_t \CE_2(t)F\|_X\,dt \leq C\|F\|_{(X, X)_{1/2, 1}}
= C\|F\|_X.
$$
Since $(X_\sigma, X_{-\sigma})_{1/2, 1} = X$ and $e^{-At}F = \CE_1(t)F + \CE_2(t)F$, we have 
$$\int^\infty_0 e^{-\gamma t}\|\pd_te^{-At}F\|_X\,\d t \leq C\|F\|_X,$$
which is our $L_1$ maximal regularity. \par
After reformulating equations \eqref{ns:1} by using
 Lagrange transformation to eliminate the convection term $\bv\cdot\nabla\rho$, 
we apply the Banach fixed point theorem to the resultant nonlinear problem based on
our $L_1$-$X$ maximal regularity to prove the local well-posedness in  Lagrange coordinates.
This is a rough idea of our proof of Theorem \ref{thm:1}.

\subsection{Problem Reformulation}\label{sec.1.1}
To prove Theorem \ref{thm:1}, it is advantageous to transfer 
equations \eqref{ns:1} to equations in  Lagrange coordinates.
In fact, 
the convection term $\bv\cdot\nabla\rho$ in the material derivative disappears
in the equations of Lagrange coordinates. \par
Let $\bu(x, t)$ be the velocity field in  Lagrange coordinates: 
$x=(x_1, \ldots, x_N)$ and we consider  Lagrange transformation:
$$y = X_\bu(x, t) := x + \int^t_0 \bu(x, \tau)\,\d\tau,
$$
where equations \eqref{ns:1} are written in Euler coordinates: $y=(y_1, \ldots, 
y_N)$. We assume that 
\begin{equation}\label{assump:2}
\Bigl\|\int^T_0 \nabla \bu(\cdot, \tau)\,\d\tau\Bigr\|_{L_\infty(\HS)}
\leq c_0
\end{equation}
with some small constant $c_0>0$, and then for each $t \in (0, T)$, 
the map: $X_\bu(x, t)  = y$ is a 
$C^1$ diffeomorphism from $\HS$ onto $\Phi(\HS)$ under the assumption that 
$\bu \in L_1((0, T), B^{s+2}_{q,1}(\HS)^N)$
with $-1 + N/q \leq s < 1/q$ (cf. Danchin et al \cite{DHMT}).  Moreover,
using an argument due to Str\"ohmer \cite{Str:89}, we have $\Phi(\HS) = \HS$, 
and so as a conclusion, $\Phi(\HS)$ is a $C^1$ diffeomorphism from 
$\HS$ onto $\HS$.\par 
We shall drive equations in Lagrange coordinates. Let 
$\BA_\bu$ is the Jacobi matrix of transformation: $y = X_\bu$, that is
$$\BA_\bu = \frac{\pd x}{\pd y} =(\frac{\pd y}{\pd x})^{-1} 
=\Bigl(\BI + \int^t_0\nabla \bu(x, \tau)\,\d\tau\Bigr)^{-1}
= \sum_{j=0}^\infty \Bigl(\int^t_0\nabla \bu(x, \tau)\,\d\tau\Bigr)^j, $$
which is well-defined under the smallness assumption \eqref{assump:2}, 
where $\BI$ denotes the $N\times N$ identity matrix. 
We have the following 
well-known formulas:
\begin{equation}\label{trans:1}\begin{aligned}
\nabla_y & = \BA_{\bu}^\top \nabla_x, \quad
\dv_y (\,\cdot\,)  = \BA_{\bu}^\top \colon \nabla_x (\,\cdot\,)
= \dv_x(\BA_{\bu} (\,\cdot \, )), \\
\nabla_y \dv_y (\, \cdot \,) & = 
\BA_{\bu}^\top \nabla_x ((\BA_{\bu}^\top - \BI) \colon \nabla_x (\, \cdot \, ))
+ \BA_{\bu}^\top \nabla_x \dv_x (\, \cdot \,), \\
\Delta_y (\,\cdot\,) &= \dv_y\nabla_y(\, \cdot \,) 
= \dv_x(\BA_{\bu} \BA_{\bu}^\top\nabla_x (\, \cdot \,))
= \dv_x((\BA_{\bu} \BA_{\bu}^\top-\BI)\nabla_x (\cdot)) + \Delta_x (\, \cdot \,).
\end{aligned}\end{equation}
 Transformation law \eqref{trans:1}
transforms the system of equations \eqref{ns:1} into the following system of equations: 
\begin{equation}\label{ns:2}\left\{\begin{aligned}
\pd_t\rho+\rho\dv\bu = F(\rho, \bu)& 
&\quad&\text{in $\HS\times(0, T)$}, \\
\rho\pd_t\bu - \alpha\Delta \bu  -\beta\nabla\dv\bu 
+  \nabla P(\rho)
 = 
 \bG(\rho, \bu)& &\quad&\text{in $\HS\times(0, T)$}, \\
\bu|_{\pd\HS} =0, \quad (\rho, \bu)|_{t=0} = (\rho_0, \bu_0)
& &\quad&\text{in $\HS$}.
\end{aligned}\right.\end{equation}
Here, we have set
\begin{equation}\label{term:1}\begin{aligned}
F(\rho, \bu) & =  \rho((\BI-\BA_\bu):\nabla\bu) 
\\
\bG(\rho, \bu)& = 
(\BI-(\BA_\bu^\top)^{-1})(\rho\pd_t\bu -\alpha\Delta\bu)  
 +\alpha(\BA_\bu^\top)^{-1}\dv((\BA_\bu\BA_\bu^\top-\BI):\nabla\bu)\\
&+\beta\nabla((\BA_\bu^\top-\BI):\nabla\bu).
\end{aligned}\end{equation}

For equations \eqref{ns:2}, we shall prove the following theorem.
\begin{thm}\label{thm:2} Let $1 < q < \infty$ and $-1+1/q \leq s < 1/q$. Assume that 
 $s$ satisfies \eqref{assump:s}. 
Let $\eta_0 = \rho_* + \tilde\eta_0$ be a given initial data such that 
$\tilde\eta_0 \in B^{s+1}_{q,1}(\HS)$ and for some
 positive constants $\rho_1$ and $\rho_2$, the assumption \eqref{assump:0}
holds. 
Then, there exist constants $\delta>0$ and $T>0$ such that 
for any initial data $\rho_0 \in B^{s+1}_{q,1}(\HS)$
and $\bu_0 \in B^s_{q,1}(\HS)^N$ satisfying the compatibility condition:
$\bu_0|_{\pd\HS} =0$, and  $\|\rho_0 - \eta_0\|_{B^{s+1}_{q,1}(\HS)}
\leq \sigma$, problem \eqref{ns:2} admits 
 unique solutions $\rho$ and $\bu$ satisfying
the regularity conditions:
$$\rho-\rho_0 \in W^1_1((0, T), B^{s+1}_{q,1}(\HS)), 
\quad 
\bu \in L_1((0, T), B^{s+2}_{q,1}(\HS)^N) \cap 
W^1_1((0, T), B^s_{q,1}(\HS)^N). 
$$
\end{thm}
\subsection{Short History} The mathematical study of compressible viscous fluids has a long history since
1950's. In fact, the first result was a uniquness theorem prove by Graffi
\cite{Grafi} and Serrin \cite{Serrin}. A local in time existence theorem was proved by Nash \cite{Nash62},
Itaya \cite{Itaya71} and Vol'pert and Hudjaev \cite{VH} in $\BR^3$ in the H\"older continuous function space.
After these works by pioneers,  much study has been done with the development of modern 
mathematics. 
We do not aim to give an extensive list of references, but refer to the following references and references 
given therein only for  unique existence theorems of strong solutions. \par 
A local in time unique existence thoerem
was proved by Solonnikov \cite{Sol80} in $W^{2,1}_q$ with $N < q < \infty$, by Tani \cite{Tani77} in the 
H\"older spaces, by Str\"ohmer \cite{Str:89} with analytic semigroup approach
and by Enomoto and Shibata \cite{ES}   in the $L_p$-$L_q$ maximal regularity class, where $\CR$ boundedness
of solution operators have been used.  If the fluid domain is $\BR^N$, 
the local well-posedness was proved by Charve and Danchin \cite{CD10} 
in the $L_1$ in time framework. 
\par
A global well-posedness was proved by
Matsumura and Nishida \cite{MN80, MN83} by energy methods and refer
to the survey paper by Shibata and Enomoto \cite{SE18} for several extensions of Matsumura and 
Nishida's work and the optimal decay properties of solutions  in the whole space and exterior domains. 
The global well-posedness in the $L_1$ in time framework was proved by Danchin \cite{D08} and also 
see Charve and Danchin \cite{CD10}, Abidi and Paicu \cite{AP07} and Haspot \cite{H11}. 
The global well-posedness in the $L_q$ maximal regularity framework ($1 < q < \infty$) 
 was proved by
Mucha and Zajaczkowski \cite{MZ}  and 
 in the $L_p$ in time and $L_q$ in space 
maximal regularity framework ($1 < p, q < \infty$) by Shibata \cite{S22}.
 Kagei and Kobayashi \cite{KK02, KK05} proved the global well-posedness
with optimal decay rate in the half-space and by Kagei \cite{Kagei08} in the layer domain.
Periodic solutions were treated by Valli \cite{Vali83}, Tsuda \cite{Tsuda16} and 
references given  therein. 
\subsection{Notation} The symbols $\BN$, $\BR$ and $\BC$ denote the set of all natural numbers, real numbers and 
complex numbers. Set $\BN_0 = \BN \cup\{0\}$. Let 
$L_q(\Omega)$, $W^m_q(\Omega)$ and $B^s_{q,r}(\Omega)$ 
denote the standard 
Lebesgue space, Sobolev space, and Besov space definded on a domain $\Omega$ in $N$ dimensional Euclidean space
$\BR^N$, while $\|\cdot\|_{L_q(\Omega)}$, $\|\cdot\|_{W^m_q(\Omega)}$, and $\|\cdot\|_{B^s_{q,r}(\Omega)}$ denote
their norms.  For time interval $I$, $L_q(I, X)$ and $W^1_q(I, X)$ denote respective $X$-valued Lebsgue space and 
Sobolev space of order 1.  $W^\alpha_q(I, X) = (L_q(I, X), W^1_q(I, X))_{\alpha, q}$, where $(\cdot, \cdot)_{\theta, r}$
denote  real interpolation functors for $\theta \in (0, 1)$ and $1 \leq r \leq\infty$. 
For $1\leq q < \infty$, we write
$$
\|f\|_{L_q(I, X)} = \Bigl(\int_I\|f(t)\|_X^q\,\d t\Bigr)^{1/q}, \quad \|e^{-\gamma t}f\|_{L_q(I, X)}
= \Bigl(\int_I (e^{-\gamma t}\|f(t)\|_X)^q\,\d t\Bigr)^{1/q}.
$$
Let 
$BC^0(I, X)$ denote the set of all  $X$-valued bounded continuous 
functions defined on $I$.  For any integer $m \geq 1$, $BC^m(I, X)$ denotes the set of all $X$-valued 
bounded continuous functions whose derivatives exist and bounded in $I$  up to order $m$. Set 
$$\|f\|_{BC(I, X)} = \sup_{t \in I}\|f(t)\|_X, \quad \|f\|_{BC^m(I,X)} = \|f\|_{BC^0(I, X)}
+ \sum_{j=1}^m \sup_{t\in I}\|(D_t^jf)(t)\|_X.$$

For the differentiation, $D^\alpha f :=\pd_x^\alpha f = \pd^{|\alpha|}f/{\pd x_1^{\alpha_1} \cdots \pd x_N^{\alpha_N}}
$ for multi-index $\alpha=(\alpha_1, \ldots, \alpha_N)$ with  $|\alpha| = \alpha_1+\cdots + \alpha_N$. 
For the notational simplicity, we write $\nabla f = \{ \pd_x^\alpha f \mid|\alpha|=1\}$, $\nabla^2 f
= \{\pd_x^\alpha f\mid |\alpha|=2\}$, $\bar\nabla f=(f, \nabla f)$, $\bar\nabla^2 f = 
(f, \nabla f, \nabla^2 f)$. For a Banach space $X$,  $\CL(X)$ denotes
the set of all bounded linear operators from $X$ into itself and 
$\|\cdot\|_{\CL(X)}$ denotes its norm.    
Let $\bI$ denote the identity operator
and $\BI$ the $N\times N$ identity matrix.  For $\mu \in (0, \pi/2)$, 
$$\Sigma_\mu = \{\lambda \in \BC\setminus\{0\} \mid |\arg\lambda| \leq \pi-\mu\}.$$
For any Banach space $X$ with norm $\|\cdot\|_X$, $X^N = \{\bff = (f_1, \ldots, f_N) \mid f_i \in X \enskip(i=1, \ldots, N)\}$
and $\|\bff\|_X = \sum_{i=1}^N \|f_i\|_X$. For a vector $\bv$ and a matrix $\BA$, $\bv^\top$ and $\BA^\top$ denote
respective the transpose of $\bv$ and the transpose of $\BA$.  \par
The letter $C$ denotes a generic constant and $C_{a,b,\cdots} = C(a, b, \cdots)$ denotes the constant depending on
quantities $a$, $b$, $\cdots$. $C$, $C_{a,b,\cdots}$, and $C(a, b, \cdots)$ may change from line to line. 
\section{Spectral Analysis}\label{sec.2}

Let $\tilde\eta_0 \in B^{s+1}_{q,1}(\HS)$ and set $\eta_0(x) = \gamma_* + \tilde\eta_0(x)$.
Let $\eta^\epsilon_0=\gamma_*+\tilde\eta^\epsilon_0$ 
is a regularization of $\eta_0$ satisfying the following 
conditions:
\begin{equation}
\label{appro:1.1}
\lim_{\epsilon\to0}\|\tilde\eta^\epsilon_0-\tilde\eta_0\|_{B^{s+1}_{q,1}(\HS)} =0. 
\end{equation}
For any $\epsilon>0$, there exists a constant
$C_\epsilon>0$ such that 
\begin{equation}
\label{appro:1.2}
\|\nabla \tilde\eta^\epsilon_0\|_{B^{N/q  }_{q,1}(\HS)} \leq C_\epsilon\|\tilde\eta_0\|_{B^{N/q}_{q,1}(\HS)}.
\end{equation}
If $N/q \leq s$, then $C_\epsilon$ is a constant independent of $\epsilon$, but 
if $N/q > s$, then $C_\epsilon$ is a constant  such that $\lim_{\epsilon\to0} C_\epsilon=\infty$.   
\par   
In fact, let $\kappa \in C^\infty_0(\BR^N)$ such that ${\rm supp}\, \kappa \subset 
\{x \in \BR^N \mid |x| \leq 1\}$ and $\int_{\BR^N} \kappa(x)\,\d x=1$, and set
$\kappa_\epsilon(x) = \epsilon^{-N}\kappa(x/\epsilon)$.   Let $\tilde\zeta_0$ be an 
extension of $\tilde\eta_0$ to $\WS$ such that $\tilde\zeta_0|_{\HS}=\tilde\eta_0$ in $\HS$
and $\|\tilde\zeta_0\|_{B^{s+1}_{q,1}(\WS)} \leq C\|\tilde\eta_0\|_{B^{s+1}_{q,1}(\HS)}$.
And then, we define 
\begin{equation}\label{conv:1}
\tilde\eta^\epsilon_0(x) = \int_{\WS} \kappa_\epsilon(x-y)\tilde\zeta_0(y)\,\d y.
\end{equation}
We see easily that this $\tilde\eta^\epsilon_0$ satisfies \eqref{appro:1.1} and \eqref{appro:1.2}.

 In this section,  we consider a generalized resolvent problem: 
\begin{equation}\label{s:2}\left\{\begin{aligned}
\lambda \rho+\eta_0^\epsilon\dv\bv = f& 
&\quad&\text{in $\HS$}, \\
\eta_0^\epsilon\lambda\bv - \alpha\Delta \bv  -\beta\nabla\dv\bv 
+  \nabla(P'(\eta_0^\epsilon) \rho) = \bg& &\quad&\text{in $\HS$}, \\
\bv|_{\pd\HS} =0.
\end{aligned}\right.\end{equation}

Let $\Sigma_\mu$ be the set defined in \eqref{resol:0} for $\mu \in (0, \pi/2)$. 
 Below, $\mu \in (0, \pi/2)$ is fixed. 
We shall prove the following theorem.
\begin{thm}\label{thm:3}
Let $1 < q < \infty$ and $-1 + 1/q  < s < 1/q$.  Assume that $s$ satisfies 
\eqref{assump:s}. 
Let $\eta_0(x)= \gamma_* + \tilde\eta_0(x)$ and assume that $\tilde\eta_0 \in B^{s+1}_{q,1}(\HS)$. 
Let $\tilde\eta^\epsilon_0$ be a 
regularization of $\tilde\eta_0$ satisfying assumptions \eqref{appro:1.1}
and \eqref{appro:1.2} and set $\eta^\epsilon_0 = \gamma_* + \tilde\eta^\epsilon_0.$
Then, the following three assertions hold. \par
\thetag1 ~ There exist 
constants $\gamma>0$ and $C$ 
such that for any $\lambda \in \Sigma_\mu+ \gamma$,
$f \in B^{s+1}_{q,1}(\HS)$ and $\bg \in B^s_{q,1}(\HS)^N$, 
problem \eqref{s:2} admits unique solutions 
$\rho \in B^{s+1}_{q,1}(\HS)$ and $\bv \in B^{s+2}_{q,1}(\HS)^N$
satsifying the estimate:
\begin{equation}\label{est:1.0}
\|(\lambda, \lambda^{1/2}\bar\nabla, \bar\nabla^2)
\bv\|_{B^{s}_{q,1}(\HS)} +\|\lambda\rho\|_{B^{s+1}_{q,1}(\HS)}
\leq C(\|f\|_{B^{s+1}_{q,1}(\HS)} + \|\bg\|_{B^s_{q,1}(\HS)})
\end{equation}
for every $\lambda \in \Sigma_\mu + \gamma$. 
\par
\thetag2~Let $\sigma>0$ be a small number such that 
$-1+1/q < s-\sigma < s+\sigma < 1/q$. Assume that 
\begin{equation}\label{assump:sigma}\left\{\begin{aligned}
s-\sigma& > 0&\quad&\text{when $s>0$}, \\
1 + \sigma& < \frac{2N}{q} &\quad&\text{when $N-1 < q < 2N$ and $s\leq 0$}, \\
|s| + \sigma &< \frac{N}{q} &\quad&\text{when $q \geq 2N$ and $s\leq 0$}.
\end{aligned}\right.\end{equation}
Then, there exist
constants $\gamma$ and $C$ such that 
for every $\lambda \in \Sigma_\mu + \gamma$ there hold 
\begin{equation}
\|(\lambda, \lambda^{1/2}\bar\nabla, \bar\nabla^2)\bv\|_{B^s_{q,1}(\HS)}
\leq C|\lambda|^{-\frac{\sigma}{2}}(\|f\|_{B^{s+1+\sigma}_{q,1}(\HS)} + 
\|\bg\|_{B^s_{q,1}(\HS)})
\label{fundest.2**} \end{equation}
provided $f \in B^{s+1+\sigma}_{q,1}(\HS)$ and $\bg \in B^{s+\sigma}_{q,1}(\HS)^N$ additionally, as well as 
\begin{equation}
\|(\lambda, \lambda^{1/2}\bar\nabla, \bar\nabla^2)\pd_\lambda\bv\|_{B^s_{q,1}(\HS)}
\leq C|\lambda|^{-(1-\frac{\sigma}{2})}(\|f\|_{B^{s+1-\sigma}_{q,1}(\HS)}
+ \|\bg\|_{B^{s-\sigma}_{q,1}(\HS)}) 
\label{fundest.3**}
\end{equation}
provided $f \in B^{s+1-\sigma}_{q,1}(\HS)$ and $\bg \in B^{s-\sigma}_{q,1}(\HS)^N$ additionally. \par 
\thetag3~ Let $\sigma>0$ be the same small constant as in \thetag2.
Then, there exist
constants $\gamma$ and $C$ such that for every $\lambda \in \Sigma_\mu + \gamma$ we have
\begin{equation}\label{rho:1}\begin{aligned}
\|\rho\|_{B^{s+1}_{q,1}(\HS)} &\leq C|\lambda|^{-\frac{\sigma}{2}}
(\|f\|_{B^{s+1}_{q,1}(\HS)} + \|\bg\|_{B^s_{q,1}(\HS)}),  \\
\|\rho\|_{B^{s+1}_{q,1}(\HS)} &\leq C|\lambda|^{-(1-\frac{\sigma}{2})}
(\|f\|_{B^{s+1}_{q,1}(\HS)} + \|\bg\|_{B^s_{q,1}(\HS)}).
\end{aligned}\end{equation}
In the statement of \thetag1,  \thetag2 and \thetag3, 
the constant $\gamma$ depends on $\gamma_*$, $\|\tilde\eta_0\|_{B^{N/q}_{q,1}}$, and $\|\nabla\tilde\eta^\epsilon_0
\|_{B^{N/q}_{q,1}(\HS)}$, 
and the  $C$  $\gamma_*$ and $\|\tilde\eta_0\|_{B^{N/q}_{q,1}}$.
\end{thm}

In the sequel, we shall prove Theorem \ref{thm:3} as a perturbation from 
Lam\'e equations which read 
\begin{equation}\label{lame:2}
\eta_0(x)\lambda \bv - \alpha\Delta\bv -\beta\nabla\dv\bv =
\bg \quad \text{in $\HS$},  \quad \bv|_{\pd\HS} =0.
\end{equation}
for spectral parameter $\lambda \in \Sigma_\mu + \gamma$ with large enough $\gamma>0$.
Thus, we start with the existence theorem for equations \eqref{lame:2}.
\begin{thm}\label{thm:4.0}
Let $1 < q < \infty$, $-1 + 1/q < s < 1/q$, and $\sigma>0$.  Assume that $s$ and $\sigma$
satisfy \eqref{assump:s}
and \eqref{assump:sigma}, respectively.   Let $\nu = s$, or $s\pm\sigma$.  
Assume that $\tilde\eta_0 \in B^{N/q}_{q,1}$ only.   
Then, there exist constants 
$\gamma > 0$ and $C>0$ depending on $s$, $\sigma$, and $\|\tilde\eta_0\|_{B^{N/q}_{q,1}(\HS)}$
such that for any $\lambda \in \Sigma_\mu+ \gamma$
 and $\bg \in B^\nu_{q,1}(\HS)$, 
problem \eqref{lame:2} admits a unique solution
 $\bv \in B^\nu_{q,1}(\HS)^N$
satisfying  the estimate:
\begin{equation}\label{est:1.1}
\|(\lambda, \lambda^{1/2}\bar\nabla, \bar\nabla^2)
\bv\|_{B^{\nu}_{q,1}(\HS)} 
\leq C\|\bg\|_{B^\nu_{q,1}(\HS)}.
\end{equation}
\end{thm}
Before starting the proof of Theorem \ref{thm:4.0}, we show a lemma concerning the Besov norm estimates
of the product of functions.
\begin{lem}\label{lem:APH} Let $1 < q < \infty$, $-1+1/q < s < 1/q$ and $\sigma>0$.
Assume that $s$ and $\sigma$ satisfy conditions  \eqref{assump:s} and \eqref{assump:sigma},
respectively.  
Let $\nu = s$ or $s\pm\sigma$.  Then, for any $u \in B^\nu_{q,1}(\HS)$ and $v \in B^{N/q}_{q,1}(\HS)$
there holds
\begin{equation}\label{besovprod:1}
\|uv\|_{B^\nu_{q,1}(\HS)} \leq C_\nu\|u\|_{B^\nu_{q,1}(\HS)}\|v\|_{B^{N/q}_{q,1}(\HS)}.
\end{equation}
\end{lem}
\begin{proof} To prove this lemma, we use the following lemma  which follows from 
Abidi and Paicu \cite[Corollaire 2.5]{AP07}. Note that  \cite[Corollaire 2.5]{AP07} was proved
 in the homogeneous Besov spaces case 
originally but it holds also in the inhomogeneous Besov spaces by the consideration in  Haspot \cite{H11}. 
\begin{lem}\label{lem:APH*} Let $1 < q < \infty$ and $-1+1/q < s < 1/q$. 
Let $N < \beta < \infty$.  Let $\delta \geq 0$ and 
$q \leq \beta < Nq$ be numbers such that 
$$\delta \leq \frac{N}{\beta}-\frac1q.$$
If $\beta \geq q'$, assume additionally
\begin{equation}\label{prodcond:1}
\beta < \frac{N}{|s|} \quad\text{if $s<0$}. 
\end{equation}
Then, there holds 
\begin{equation}\label{prodest:1}\|uv\|_{B^s_{q,1}} 
\leq C\|u\|_{B^{s+\delta}_{q,1}}\|v\|_{B^{N/\beta-\delta}_{\beta, \infty} \cap L_\infty}
\end{equation}
for some constant $C>0$. \par
In particular, there holds 
\begin{equation}\label{prodest:1*}
\|uv\|_{B^s_{q,1}} \leq C\|u\|_{B^s_{q,1}}\|v\|_{B^{N/q}_{q,1}}.
\end{equation}
\end{lem}
\begin{remark}\label{rem:1} 
If  \eqref{prodest:1} holds with $\delta=0$ and $\beta \geq q$,  then \eqref{prodest:1*} holds. In fact, 
$B^{N/q}_{q,1}$ is continously imbedded into $L_\infty$ and so $\|v\|_{L_\infty} \leq C\|v\|_{B^{N/q}_{q,1}}$.
Moreover,  by  imbedding relations 	\cite[Theorem 9]{Muramatsu}, we have 
$$\|v\|_{B^{\frac{N}{\beta}}_{\beta, \infty}} \leq C\|v\|_{B^{\frac{N}{\beta} + N\left(\frac1q-
\frac1\beta\right)}_{q,1}} = C\|v\|_{B^{\frac{N}{q}}_{q,1}}.$$
Thus, $\|v\|_{B^{N/\beta}_{\beta, \infty} \cap L_\infty} \leq C\|v\|_{B^{N/q}_{q,1}}$, which shows 
\eqref{prodest:1*}.
\end{remark}
To prove Lemma \ref{lem:APH}, we use Lemma \ref{lem:APH*} with $\delta=0$ and $\beta=q$. 
Since $N \geq 2$, it holds obviously that $N/q-1/q >0$.  If $q \geq q'$ and $\nu <0$, 
then the requirement
is that $|\nu| <N/q$.  If $s > 0$, then $s\pm\sigma>0$, and so $\nu>0$.  If $s \leq 0$ and $N-1 < q < 2N$, noting
that $|s| \leq 1-N/q$, we see that 
$|\nu| \leq |s|+\sigma < 1-N/q + \sigma < N/q$ as follows from \eqref{assump:sigma}. 
If $s \leq 0$ and $q \geq 2N$, then $|\nu| \leq |s|+\sigma <N/q$ as also follows from \eqref{assump:sigma}.
Therefore, the requirements are satisfied, and so by Lemma \ref{lem:APH*} we have Lemma \ref{lem:APH}.
This completes the proof of Lemma \ref{lem:APH}. 
\end{proof}
{\bf Proof of  Theorem \ref{thm:4.0}.} To prove Theorem \ref{thm:4.0}, 
we shall construct an approximate solution 
for each point $x_0 \in \overline{\HS}$. Let $\nu = s$ or $s\pm\sigma$.    Recall that $\eta_0(x) = \gamma_*
+\tilde \eta_0(x)$ and $\tilde\eta_0 \in  B^{N/q}_{q,1}(\HS)$.
To construct an approximation solution, we use a theorem for unique existence of solutions
of  the constant coefficient Lam\'e equations which read 
\begin{equation}\label{fund:1}
\gamma_0\lambda\bv -\alpha\Delta \bv - \beta\nabla\dv\bv 
=\bg\quad\text{in $\HS$}, \quad 
\bv|_{\pd\HS}=0.
\end{equation}
From Kuo \cite{Kuo23} the following theorem follows.
\begin{thm}\label{thm:kuo} Let $1 < q < \infty$ and $-1+1/q < \nu < 1/q$. 
Assume that $\alpha$ and $\beta$ are 
constants satisfying the conditions:
\begin{equation}\label{assump:1}
\alpha >0, \quad \alpha + \beta>0.\end{equation}
Moreover, we assume that there exist positive constants $M_1$ and $M_2$ such that 
$$M_1 \leq \gamma_0  \leq M_2.$$ 
Then, there exists a $\gamma > 0$ independent of $\gamma_0$
 such that for any $\lambda \in 
\Sigma_\mu + \gamma$ and  $\bg \in B^\nu_{q,1}(\HS)$, 
problem \eqref{fund:1} admits a unique solution
$\bv \in B^{\nu+2}_{q,1}(\HS)$ satisfying
the estimate:
\begin{equation}\label{fundest.1} 
\|(\lambda, \lambda^{1/2}\bar\nabla,
\bar\nabla^2)\bv\|_{B^\nu_{q,1}(\HS)} \leq C\|\bg\|_{B^\nu_{q,1}(\HS)}
\end{equation}
for some constant $C$.  \par
Moreover, let $-1+1/q < s < 1/q$ and let $\sigma>0$ be a small positive constant such that 
$-1+1/q < s-\sigma < s < s+\sigma < 1/q$.  Then, for any $\lambda \in \Sigma_\mu + \gamma$ and 
$\bg \in B^{s\pm\sigma}_{q,1}(\HS) \cap B^s_{q,1}(\HS)$, 
a solution $\bv \in B^{s\pm\sigma+2}_{q,1}(\HS) \cap B^{s+2}_{q,1}(\HS)$ of equations \eqref{fund:1} satisfies 
the following estimates: 
\begin{align}
\|(\lambda, \lambda^{1/2}\bar\nabla, \bar\nabla^2)\bv\|_{B^s_{q,1}(\HS)}
&\leq C|\lambda|^{-\frac{\sigma}{2}} \|\bg\|_{B^{s+\sigma}_{q,1}(\HS)}, 
\label{fundest.2} \\
\|(\lambda, \lambda^{1/2}\bar\nabla, \bar\nabla^2)\pd_\lambda\bv\|_{B^s_{q,1}(\HS)}
&\leq C|\lambda|^{-(1-\frac{\sigma}{2})} \|\bg\|_{B^{s-\sigma}_{q,1}(\HS)}.
\label{fundest.3}
\end{align} 
\par
Here, the constants $\gamma$ and $C$ depend on $M_1$,  $M_2$, and $\nu$, 
but  independent
of $\gamma_0$  as far as the assumption \eqref{assump:1} holds. 
\end{thm}
\begin{remark} (1) The same assertions hold for the  whole space problem: 
\begin{equation}\label{fund:2}\begin{aligned}
\gamma_0\lambda\bv -\alpha\Delta \bv - \beta\nabla\dv\bv 
=\bg&&\quad&\text{in $\WS$}.
\end{aligned}\end{equation}
(2) For any $\lambda \in \Sigma_\mu + \gamma$ and $\bg \in B^s_{q,1}(\HS)$, there holds
\begin{equation}
\label{fundest:3*}
\|\bv\|_{B^s_{q,1}(\HS)} \leq C|\lambda|^{-(1-\frac{\sigma}{2})}\|\bg\|_{B^{s-\sigma}_{q,1}(\HS)}.
\end{equation}
In fact, we write resolvent by $\CS_0(\lambda)\bg$, which is holomorphic function with respect
to $\lambda \in\Sigma_\mu + \gamma$.  In fact, $\CS_0(\lambda)\bg=\bv$.  Differentiating
equations \eqref{fund:1} with respect to $\lambda$, we have
\begin{equation}\label{fund:1*}
\gamma_0\lambda \pd_\lambda - \alpha\Delta \pd_\lambda\bv - \beta\nabla\dv\pd_\lambda\bv
= -\gamma_0\bv \quad\text{in $\HS$}, \quad \pd_\lambda\bv|_{\pd\HS} = 0.
\end{equation}
Thus, we have $\pd_\lambda\bv = -\gamma_0\CS_0(\lambda)\bv 
= -\gamma_0\CS_0(\lambda)\CS_0(\lambda)\bg$. Let $D^{s+2}_{q,1}(\HS) = \{\bu \in B^{s+2}_{q,1}(\HS) \mid
\bu|_{\pd\HS}=0\}$. Since $\CS_0(\lambda)$ is a surjective map from $B^s_{q,1}(\HS)$ onto 
$D~{s+2}_{q,1}(\HS)$, and so the inverse map $\CS_0(\lambda)^{-1}$ exists and it is a surjective
map from $D^{s+2}_{q,1}(\HS)$ onto $B^s_{q,1}(\HS)$.  Thus, $\bv
= -\gamma_0^{-1}\CS_0(\lambda)\pd_\lambda\bv$. By \eqref{fundest.3}, we have
$$\|\bv\|_{B^s_{q,1}(\HS)} \leq C\|\bar\nabla^2\pd_\lambda\bv\|_{B^s_{q,1}} 
\leq C|\lambda|^{-(1-\frac{\sigma}{2})}\|\bg\|_{B^{s-\sigma}_{q,1}(\HS)},
$$
which shows \eqref{fundest:3*}.  From this consideration it follows that 
\eqref{fundest.3} and \eqref{fundest.3*} is equivalent.
\end{remark}
\begin{proof} When $\gamma_0=1$, by a result due to Kuo \cite{Kuo23} 
there exist constants $C$ and $\tilde\gamma$ such that 
the existense of solutions and \eqref{fundest.1}--\eqref{fundest.3} hold. 
Here, 
the constants $C$ and $\tilde\gamma>0$ depend only on $\alpha$ and $\beta$.   
 When $\gamma_0\not=1$,  the existense of solutions and estimates
\eqref{fundest.1}--\eqref{fundest.3} hold,  replacing $\lambda$ with $\gamma_0\lambda$,
provided that  $\gamma_0\lambda \in \Sigma_\mu+ \tilde\gamma$. Since
$M_1 \leq \gamma_0 \leq M_2$, we see that $M_1|\lambda| \leq |\gamma_0\lambda| \leq M_2|\lambda|$. 
Thus, choosing $\gamma = \tilde\gamma M_1^{-1}$, we see that
$\gamma_0\lambda \in \Sigma_\mu + \tilde\gamma$ when
$\lambda \in \Sigma_\mu + \gamma$. From this consideration,
Theorem \ref{thm:kuo} follows from the $\gamma_1=1$ case. 
Here, the constants $C$ and $\gamma$ depend on $\alpha$, $\beta$, $M_1$ and $M_2$.
\end{proof}

We continue the proof of Theorem \ref{thm:4.0}. First we consider the case where $x_0 \in \pd\HS$. 
We write 
$$B_d(x_0) = \{x \in \BR^N \mid |x-x_0| \leq d\}, 
\quad B_d = B_d(0).$$
 Let  
$\varphi \in C^\infty_0(B_2(0))$ and $\psi \in C^\infty_0(B_3(0))$
such that 
 $\varphi(x)=1$ for 
$x \in B_1(0)$ and $\psi(x) =1$ for $x \in B_{2}(0)$
 and set $\varphi_{x_0}(x) = \varphi((x-x_0)/d)$ and $\psi_{x_0}(x)=
\psi((x-x_0)/d)$.  Notice that $\varphi_{x_0}(x)=1$ for $x \in B_d(x_0)$ 
and $=0$ for $x \not\in B_{2d}(x_0)$ and that $\psi_{x_0}(x) = 1$ on
${\rm supp}\,\varphi_{x_0}$ and $\psi_{x_0}(x)=0$ for $x \not\in B_{3d}(x_0)$.
In particular, $\varphi_{x_0}\psi_{x_0} = \varphi_{x_0}$.  \par
Let $\bv \in B^s_{q,1}(\HS)^N$ be a solution of equations:
\begin{equation}\label{st:1}
\eta_0(x_0)\lambda\bv - \alpha\Delta\bv - \beta\nabla\dv\bv
 = \bg\quad\text{in $\HS$}, \quad 
\bv|_{\pd\HS} =0.
\end{equation}
For simplicity, we omit $\HS$ for the description of function spaces and thier norms like
$B^\nu_{q,1}=B^{\nu}_{q,1}(\HS)$ and $\|\cdot\|_{B^\nu_{q,1}} = \|\cdot\|_{B^\nu_{q,1}(\HS)}$
in what follows. We define an operator 
$\bT_{x_0}(\lambda)$ acting on $\bg \in B^\nu_{q,1}$ by
$\bv = \bT_{x_0}(\lambda)\bg$.
By \eqref{assump:0}, $\rho_1 < \eta_0(x_0) < \rho_2$, and so 
 by Theorem \ref{thm:kuo} there exist constants
$C$ and $\gamma$ independent of $x_0$ such that 
\begin{equation}\label{est:f1}
\|(\lambda, \lambda^{1/2}\bar\nabla, \bar\nabla^2)\bT_{x_0}(\lambda)\bg\|_{B^\nu_{q,1}}
\leq C\|\bg\|_{B^\nu_{q,1}}
\end{equation}
for every $\lambda \in \Sigma_\mu + \gamma$. 
 Let $A_{x_0} = \eta_0(x_0)
+ \psi_{x_0}(x)(\eta_0(x)-\eta_0(x_0))$.
And then, $\bv$ satisfy the following
equations:
\begin{equation}\label{st:2}
A_{x_0}\lambda\bv - \alpha\Delta\bv - \beta\nabla\dv\bv
 = \bg
+ \bS_{x_0}(\lambda)\bg\quad\text{in $\HS$}, \quad 
\bv|_{\pd\HS} =0.
\end{equation}
Here, we have set 
\begin{align*}
\bS_{x_0}(\lambda)\bg & = 
\psi_{x_0}(x)(\eta_0(x)-\eta_0(x_0))\lambda\bv.
\end{align*}

We now estimate $\psi_{x_0}(\eta_0(x_0)-\eta_0(x))\lambda\bv$.  
Note that $\eta_0(x)-\eta_0(x_0) = \tilde\eta_0(x)-\tilde\eta_0(x_0)$. 
By Lemma \ref{lem:APH}, we have
\begin{equation}\label{fundest:1}
\|\psi_{x_0}(\eta_0(x_0)-\eta_0(x))\lambda\bv\|_{B^{\nu}_{q,1}} \leq 
C\|\psi_{x_0}(\tilde\eta_0(x_0)-\tilde\eta_0(x))
\|_{B^{N/q}_{q,1}}\|\lambda\bv\|_{B^\nu_{q,1}}. 
\end{equation}
To estimate $\|\psi_{x_0}(\eta_0(x_0)-\eta_0(x))\|_{B^{N/q}_{q,1}}$, 
we use the following lemma  due to Danchin-Tolksdorf \cite[Proposition B.1]{DT22}.
\begin{lem}\label{Prop:B.1} Let $f \in B^{N/q}_{q,1}(\HS)$ for some $1 \leq q \leq \infty$.  
Then, 
$$
\lim_{d\to0}\|\varphi_{x_0, d}(\cdot)(f(\cdot)-f(x_0))\|_{B^{N/q}_{q,1}(\HS)}=0\quad\text{uniformly with respect to $x_0$}.
$$
\end{lem}
By Lemma \ref{Prop:B.1}, for any $\delta > 0$ there exists a $d>0$  such that 
\begin{equation}\label{small:0.1}
\|\psi_{x_0}(\eta_0(x_0)-\eta_0(x))\|_{B^{N/q}_{q,1}} \leq \delta
\end{equation}
Notice that the choice of distance $d$ is independent of $x_0$. 
From \eqref{fundest:1} and \eqref{small:0.1}, it follows that 
\begin{equation}\label{est:2}\begin{aligned}
\|\bS_{x_0}(\lambda)\bg\|_{B^\nu_{q,1}}
\leq C\delta\|\lambda\bv\|_{B^\nu_{q,1}}.
\end{aligned}\end{equation}
Choosing $d>0$ so small that $C\delta \leq 1/2$, 
we have $\|\bS_{x_0}\|_{\CL(B^\nu_{q,1})} \leq 1/2$.   
Thus, 
the inverse $(\bI + \bS_{x_0}(\lambda))^{-1}$ of the operator $\bI + \bS_{x_0}(\lambda)$
exists, 
where $\bI$ is the identity operator on $ B^\nu_{q,1}$.
Recalling the operator $\bT_{x_0}(\lambda)$ is defined 
by $\bv = \bT_{x_0}(\lambda)\bg$, 
and setting $\bw_{x_0} = \bT_{x_0}(\lambda)(\bI + \bS_{x_0}(\lambda))^{-1}\bg$, 
by \eqref{est:f1} we see that 
$\bw_{x_0}$ satisfies equations:
\begin{equation}\label{st:3*}
A_{x_0}\lambda\bw_{x_0} - \alpha\Delta\bw_{x_0} - \beta\nabla\dv
 = \bg\quad\text{in $\HS$}, \quad 
\bw|_{\pd\HS} =0, \end{equation}
as well as the estimate
\begin{equation}\label{est:m1}
\|(\lambda, \lambda^{1/2}\bar\nabla, \bar\nabla^2)\bw_{x_0}\|_{B^\nu_{q,1}}
\leq C\|(\bI+\bS_{x_0}(\lambda))^{-1}\bg\|_{B^\nu_{q,1}}
\leq C\|\bg\|_{B^\nu_{q,1}}
\end{equation}
for every $\lambda \in \Sigma_\mu + \gamma$, where $C$ is independent of 
$d$, and  $\gamma>0$ is the same as in Theorem \ref{thm:kuo}. \par 
Finally,  we set $\bv_{x_0} = \varphi_{x_0}\bw_{x_0}$.  Since 
$\psi_{x_0}\varphi_{x_0} = \varphi_{x_0}$, we have 
$A_{x_0}\varphi_{x_0} = \eta_0(x)\varphi_{x_0}$.
From \eqref{st:3*} it follows that
\begin{equation}\label{st:4*}
\eta_0(x)\lambda\bv_{x_0} - \alpha\Delta\bv_{x_0} - \beta\nabla\dv\bv_{x_0}
 = \varphi_{x_0}\bg
+ \bU_{x_0}(\lambda)\bg\quad\text{in $\HS$}, \quad 
\bv_{x_0}|_{\pd\HS} =0, 
\end{equation}
where we have set
\begin{align*}
\bU_{x_0}(\lambda)\bg & = -\alpha((\Delta\varphi_{x_0})\bw_{x_0} + 2(\nabla\varphi_{x_0})\nabla\bw_{x_0})
-\beta(\nabla((\nabla\varphi_{x_0})\cdot\bw_{x_0}) + (\nabla\varphi_{x_0})\dv\bw_{x_0}).
\end{align*}
From \eqref{est:m1}, we see that 
\begin{equation}\label{pert:2}\begin{aligned}
\|(\lambda, \lambda^{1/2}\bar\nabla, \bar\nabla^2)\bv_{x_0}\|_{B^\nu_{q,1}}
\leq C_d\|\bg\|_{B^\nu_{q,1}},
\end{aligned}\end{equation}
as well as 
\begin{equation}\label{remainder:1}
\|\bU_{x_0}(\lambda)\bg\|_{B^\nu_{q,1}} \leq C_d|\lambda|^{-1/2}\|\bg\|_{B^\nu_{q,1}}
\end{equation}
for every $\lambda \in \Sigma_\mu + \gamma$. 
Here,  $C_d$ is a constant depends solely on
$d>0$ such that $C_d\to \infty$ as $d\to0$.

\par
Next, we pick up $x_1 \in \HS$ and we choose 
$d_1>0$ such that $B_{3d_1}(x_1) \subset \HS$.
Let $\varphi_{x_1}(x) = \varphi((x-x_1)/d_1)$ and $\psi_{x_1}(x) = \psi((x-x_1)/d)$.
 Analogously  to \eqref{pert:2} and 
\eqref{remainder:1}, if we choose $d_1>0$ small enough, there exist a  
$\bw_{x_1} \in B^{s+2}_{q,1}(\WS)$ satisfying equations
\begin{equation}\label{st:4*}
A_{x_1}\lambda\bw_{x_1} - \alpha\Delta\bw_{x_1} - \beta\nabla\dv\bw_{x_1}
= \bg\quad\text{in $\HS$}, \quad 
\bw_{x_1}|_{\pd\HS} =0, 
\end{equation}
where $A_{x_1} = \eta_0(x_1) + \psi_{x_1}(\eta_0(x)
-\eta_0(x_1))$, and the estimate:
\begin{equation}\label{est:3*}
\|(\lambda, \lambda^{1/2}\bar\nabla, \bar\nabla^2)\bw_{x_1}\|_{B^\nu_{q,1}}
\leq C\|\bg\|_{B^\nu_{q,1}}.
\end{equation}
Let $\bv_{x_1}=
\varphi_{x_1}\bw_{x_1}$ and then $\bv_{x_1}$ satisfies equations:
\begin{equation}\label{st:4}
\eta_0(x)\lambda\bv_{x_1} - \alpha\Delta\bv_{x_1} - \beta\nabla\dv\bv_{x_1}
= \varphi_{x_1}\bg
+ \bU_{x_1}(\lambda)\bg\quad\text{in $\HS$}, \quad 
\bv_{x_1}|_{\pd\HS} =0, 
\end{equation}
where we have set
\begin{align*}
\bU_{x_1}(\lambda)\bg & = -\alpha((\Delta\varphi_{x_1})\bw_{x_1} + 2(\nabla\varphi_{x_1})\nabla\bw_{x_1})
-\beta(\nabla((\nabla\varphi_{x_1})\cdot\bw_{x_1}) + (\nabla\varphi_{x_1})\dv\bw_{x_1}).
\end{align*}
Moreover, by \eqref{est:3*}, we have 
\begin{gather}\label{pert:3}
\|(\lambda, \lambda^{1/2}\bar\nabla, \bar\nabla^2)\bv_{x_1}\|_{B^\nu_{q,1}}
\leq C_{d_1}\|\bg\|_{B^\nu_{q,1}}, \\
\|\bU_{x_1}\bg\|_{B^\nu_{q,1}}  \leq C_{d_1}|\lambda|^{-1/2}\|\bg\|_{B^\nu_{q,1}}
\label{remainder:2}
\end{gather}
for every $\lambda \in \Sigma_\mu + \gamma$, where $C_{d_1}$ is a constant depends solely on
$d_1>0$ such that $C_{d_1}\to \infty$ as $d_1\to0$. 


Finally, we consider  the far field case. 
Let $\tilde\psi \in C^\infty(\BR)$ which equals to $1$ for 
$|x| \geq2$ and  $0$ for $|x| \leq 1$, and set $\psi_R(x) = \tilde\psi(x/R)$. 
Let 
 $\bv$ be a  solution of equations 
\begin{equation}\label{s:3*}
\gamma_*\lambda\bv - \alpha\Delta \bv  -\beta\nabla\dv\bv 
 = \bg \quad\text{in $\HS$}, \quad 
\bv|_{\pd\HS} =0.
\end{equation}
We define an operator $\bT_R(\lambda)$ by $\bv
= \bT_R(\lambda)\bg$. By Theorem \ref{thm:kuo}, we have
\begin{equation}\label{est:2.2}
\|(\lambda, \lambda^{1/2}\bar\nabla, \bar\nabla^2)\bT_R(\lambda)\bg\|_{B^\nu_{q,1}(\HS)}
\leq C\|\bg\|_{ B^\nu_{q,1}}.
\end{equation}
 Set 
$A_R= \gamma_* + \psi_R(\eta_0(x)-\gamma_*) =\gamma_* + \psi_R\tilde\eta_0$.
By \eqref{s:3*}, we have
\begin{equation}\label{s:4*}
A_R\lambda\bv - \alpha\Delta \bv  -\beta\nabla\dv\bv 
 = \bg + \bS_{R}(\lambda)\bg\quad\text{in $\HS$}, \quad
\bv|_{\pd\HS} =0,
\end{equation}
where we have set
\begin{align*}
\bS_{R}(\lambda)\bg & = \psi_R(x)\tilde\eta_0(x)\lambda\bv.
\end{align*}
By Lemma \ref{lem:APH}, 
we have
\begin{equation}\label{fundest:2}
\|\bS_R(\lambda)\bg\|_{B^\nu_{q,1}}
\leq C\|\psi_R\tilde\eta_0\|_{B^{N/q}_{q,1}}
\|\lambda\bv\|_{B^\nu_{q,1}}.
\end{equation}
For any $\delta >0$ there exists an $R$ such that 
\begin{equation}\label{small:2}
\|\psi_R\tilde\eta_0\|_{B^{N/q}_{q,1}}  \leq \delta.
\end{equation}
This fact follows from the following lemma, the idea of whose proof is completely the same 
as in the proof of \cite[Proposition B.1]{DT22}.
\begin{lem}\label{Prop.B2} Let $f \in B^{N/q}_{q,1}(\HS)$ for some $1 \leq q \leq \infty$.  Then,
for any $\delta > 0$, there exists an $R > 1$ such that 
$$\|\psi_R f\|_{B^{N/q}_{q,1}(\HS)}  < \delta.$$
\end{lem}
\begin{proof} 
Let $m$ be an integer such that $N/q < m$.  Notice that $W^m_q(\HS)$ is dense in
$B^{N/q}_{q,1}(\HS)$.  Thus, first we assume that $f \in W^m_q(\HS)$.  Then, 
$\|f\|_{W^m_q(\HS)} < \infty$ and $\|f\|_{L_q(\HS)} < \infty$, which implies that 
for any $\delta>0$, there exists an $R > 0$ such that $\|f\|_{W^m_q(B_R^c)} < \delta$ and 
$\|f\|_{L_q(B_R^c)}  < \delta$.  Here, $B_R^c = \{x \in \BR^N \mid |x| \geq R\}$.  Thus, 
$\|\psi_Rf\|_{W^m_q(\HS)} < \delta$ and $\|\psi_Rf\|_{L_q(\HS)} < \delta$.  In fact, 
$$\|\psi_Rf\|_{W^m_q(\HS)} \leq C_m\sum_{|\beta|\leq m}R^{-(m-|\beta|)}\|D^\beta f\|_{L_q(B_R^c)}
\leq C_m\|f\|_{L_q(B_R^c)}
$$
for any $R \geq 1$ with some constant $C_m$ depending only on $m$ and $D^\alpha\tilde\psi$ ($|\alpha| \leq m$). Thus, 
choosing $R>0$ larger if necessary, we have $\|\psi_Rf\|_{W^m_q(\HS)} < \delta.$ \par

Since $\|\psi_R f\|_{B^{N/q}_{q,1}(\HS)} \leq C\|\psi_R f\|_{L_q(\HS)}^{1-\frac{N}{mq}}\|\psi_R f
\|_{W^m_q(\HS)}^{\frac{N}{mq}}$ with some constant $C$ independent of $R$ and $f$, we have
$$\|\psi_R f\|_{B^{N/q}_{q,1}(\HS)} \leq C\delta.$$
If we choose $R \geq 1$ larger, we have
$$\|\psi_R f\|_{B^{N/q}_{q,1}(\HS)} \leq \delta/2.$$
Now, in the case where $f \in B^{N/q}_{q,1}(\HS)$, we choose $g \in W^m_{q,1}(\HS)$ such that 
$$\|\psi_R(g-f) \|_{B^{N/q}_{q,1}(\HS)} < C\|g-f\|_{B^{N/q}_{q,1}(\HS)} <\delta/2.
$$
Here, $C$ is a constant indepenent of $R$.  Thus, choosing $R>0$ in such a way that 
$\|\psi_Rg\|_{B^{N/q}_{q,1}(\HS)} < \delta/2$,  we have
$$\|\psi_R f\|_{B^{N/q}_{q,1}(\HS)} \leq \|\psi_R(f-g)\|_{B^{N/q}_{q,1}(\HS)} 
+ \|\psi_Rg\|_{B^{N/q}_{q,1}(\HS)}
< \delta.$$
This completes the proof of Lemma \ref{Prop.B2}.

\end{proof}

Combining \eqref{fundest:2} and \eqref{small:2} implies 
\begin{equation}\label{small:2.3}
\|\bS_R(\lambda)\bg\|_{B^\nu_{q,1}} \leq C\delta\|\bg\|_{B^\nu_{q,1}}.
\end{equation}
Choosing $\delta>0$ in such a way that  $C\delta \leq 1/2$, we have 
$\|\bS_R(\lambda)\|_{\CL(B^\nu_{q,1})} \leq 1/2$, and so 
the inverse operator $(\bI+\bS_R(\lambda))^{-1}$ exists
for every $\lambda \in \Sigma_\mu + \gamma$.  
Thus, by \eqref{s:4*} and \eqref{est:2.2}, 
$\bw_R = \bT_R(\lambda)(\bI + \bS_R(\lambda))^{-1}\bg 
\in B^\nu_{q,1}$ satisfies equations
\begin{equation}\label{s:5*}
A_R\lambda\bw_R- \alpha\Delta \bw_R  -\beta\nabla\dv\bw_R 
 = \bg  \quad \text{in $\HS$}, \quad 
\bw_R|_{\pd\HS} =0, 
\end{equation}
as well as the estimate: 
\begin{equation}\label{est:2.3}
\|(\lambda, \lambda^{1/2}\bar\nabla, \bar\nabla^2)\bw_R\|_{B^\nu_{q,1}}
\leq C\|(\bI+\bS_R)^{-1}\bg)\|_{B^\nu_{q,1}}
\leq C\|\bg\|_{B^\nu_{q,1}}.
\end{equation}

Let $\tilde\varphi \in C^\infty(\HS)$ such that $\tilde\varphi(x) =1$
for $|x| \geq 3$ and $0$ for $|x| \leq 2$ and set $\varphi_R = \tilde\varphi(x/R)$. 
We have $\psi_R \varphi_R = \varphi_R$, and so setting
$\bv_R = \varphi_R\bw_R \in B^\nu_{q,1}(\HS)$, we see that $A_R\varphi_R\lambda\bv_R
 = \eta_0(x)\lambda\bv_R$.  Thus, 
by \eqref{s:5*} and \eqref{est:2.3}, we see that 
$\bv_R$ satisfies  the equations: 
\begin{equation}\label{s:6*}
\eta_0(x)\lambda\bv_R- \alpha\Delta \bv_R  -\beta\nabla\dv\bv_R 
 = \varphi_R\bg + \bU_{R}(\lambda)\bg
\quad\text{in $\HS$}, \quad 
\bv_R|_{\pd\HS} =0,
\end{equation}
as well as the estimate: 
\begin{equation}\label{est:2.4}
\|(\lambda, \lambda^{1/2}\bar\nabla, \bar\nabla^2)\bv_R\|_{B^\nu_{q,1}}
\leq C\|\bg\|_{B^\nu_{q,1}}
\end{equation}
for any $\lambda \in \Sigma_\mu + \gamma$.  Here,  we have set
\begin{align*}
\bU_{R}(\lambda)\bg & = -\alpha((\Delta\varphi_R)\bw_R
+2 (\nabla\varphi_R)\nabla\bw_R) - \beta(\nabla((\nabla\varphi_R)\cdot\bw_R)
+ (\nabla\varphi_R)\dv\bw_R).
\end{align*}
By \eqref{est:2.3}, we have 
\begin{equation}\label{remainder:3}
\|\bU_{R}(\lambda)\bg
\|_{B^{\nu}_{q,1}} \leq C|\lambda|^{-1/2}\|\bg\|_{B^\nu_{q,1}}.
\end{equation}
\par
Let $x^0_j \in \pd\HS$ ($j=1, \ldots, L_0$), and $x^1_j \in 
\HS$ ($j=1, \ldots, L_1$) be points such that 
$$\overline{\HS} \subset B^c_R \cup \bigcup_{j=1}^{L_0}
B_d(x^0_j) \cup \bigcup_{j=1}^{L_1} B_{d_1}(x^1_j).
$$ 
where $d> d_1 > 0$ are suitably chosen. 
Let $\psi^0_0(x) = \psi_R(x)$, $\psi^0_j(x) = \varphi((x-x^0_j)/d)$, and
 $\psi^1_j(x) = \varphi((x-x^1_j)/d_1)$, and set
$$\Psi(x) = \psi^0_0(x) + \sum_{i=0}^1\sum_{j=1}^{L_i} \psi^i_j(x).$$
We see that $\Psi(x) \geq 1$ for every $x \in \overline{\HS}$ and $\Psi \in C^\infty(\overline{\HS})$.
 Set 
$$\varphi^0_0(x) = \psi^0_0(x)/\Psi(x), \quad
\varphi^i_j(x) = \psi^i_j(x)/\Psi(x).$$
Obviously, $\varphi^0_j \in C^\infty_0(B_{2d}(x^0_j))$, 
$\varphi^1_j \in C^\infty_0(B_{2d_1}(x^1_j))$,  
$\varphi^0_0(x) = 0$ for 
$|x| \leq 2R$, and 
$$\varphi^0_0(x) + \sum_{i=0}^1\sum_{j=1}^{L_i}
\varphi^i_j(x) = 1\quad\text{for $x \in \overline{\HS}$}.
$$
Let $\bv^i_j = \bv_{x^i_j} = \varphi^i_j\bw_{x^i_j}$, 
and $\bv^0_0 = \bv_R = \varphi^0_0\bw_R$. 
Set $\bv = \bv^0_0 + \sum_{i=0}^1\sum_{j=1}^{L_i}\bv^i_j$, and then 
\begin{equation}\label{eq:t.1}
\eta_0(x)\lambda\bv -\alpha\Delta\bv-\beta\nabla\dv\bv
  = \bg
+ \bU(\lambda)\bg\quad\text{in $\HS$}, \quad 
\bv|_{\pd\HS}=0.
\end{equation}
Here, we have set 
\begin{align*}
\bU(\lambda)\bg &= -\alpha((\Delta\varphi^0_0)\bw_R +2 (\nabla\varphi^0_0)\nabla
\bw_R) -\beta(\nabla((\nabla\varphi^0_0)\cdot\bw_R) + (\nabla\varphi^0_0)\dv\bw_R )\\
& -\sum_{i=0}^1\sum_{j=1}^{L_i}\{\alpha((\Delta\nabla\varphi^i_j)\bw_{x^i_j} + 
2(\nabla\varphi^i_j)\nabla\bw_{x^i_j}) 
+\beta(\nabla((\nabla\varphi^i_j)\cdot\bw_{x^i_j})+(\nabla\varphi^i_j)\dv\bw_{x^i_j}  )\}.
\end{align*}
By \eqref{pert:2}, \eqref{pert:3}, and \eqref{est:2.4}, we have
\begin{equation}\label{main:est.1}
\|(\lambda, \lambda^{1/2}\bar\nabla, \bar\nabla^2)\bv\|_{B^\nu_{q,1}}
\leq C\|\bg\|_{B^{\nu}_{q,1}}.
\end{equation}
By \eqref{remainder:1}, \eqref{remainder:2}, and \eqref{remainder:3}, 
we have
\begin{equation}\label{remain:1}
\|\bU(\lambda)\bg\|_{B^\nu_{q,1}}
\leq C|\lambda|^{-1/2}
\|\bg\|_{B^\nu_{q,1}}
\end{equation}
for any $\lambda \in \Sigma_\mu + \gamma$. 
Choosing $\gamma>0$ so large that $C\gamma^{-1/2} \leq 1/2$, we see that 
for any $\lambda \in \Sigma_\mu + \gamma$ 
 $(\bI+\bU(\lambda))^{-1}$ exists and 
$\|(\bI+\bU(\lambda))^{-1}\|_{\CL(B^\nu_{q,1})} \leq 2$. 
If we define an operator $\bT(\lambda)$ by
$\bT(\lambda)\bg = \bv$,  by \eqref{eq:t.1}
 $\bv = \bT(\lambda)(\bI+\bU(\lambda))^{-1}\bg$ satisfies
equations:
\begin{equation}\label{eq:t.2}
\eta_0(x)\lambda\bv -\alpha\Delta\bv-\beta\nabla\dv\bv
  = \bg\quad\text{in $\HS$}, \quad 
\bv|_{\pd\HS}=0.
\end{equation}
Moreover, by \eqref{main:est.1}, we have 
$$\|(\lambda, \lambda^{1/2}\bar\nabla, \bar\nabla^2)\bT(\lambda)(\bI+\bU(\lambda))^{-1}\bg\|_{B^\nu_{q,1}}
\leq C\|(\bI+\bU(\lambda))^{-1}\bg\|_{B^\nu_{q,1}}
\leq 2C\|\bg\|_{B^\nu_{q,1}}
$$
for any $\lambda \in \Sigma_\mu + \gamma$. 
 This completes the proof of 
Theorem \ref{thm:4.0}.  \qed
\vskip0.5pc
We now consider the perturbed Lam\'e equations, which read 
\begin{equation}\label{lame:1}
\eta^\epsilon_0(x)\lambda \bv - \alpha\Delta\bv -\beta\nabla\dv\bv =
\bg \quad \text{in $\HS$},  \quad \bv|_{\pd\HS} =0.
\end{equation}
Here, $\eta^\epsilon_0$ is an approximation of $\eta_0$ given in
\eqref{appro:1.1} and \eqref{appro:1.2}. For equations \eqref{lame:1}
using Theorems \ref{thm:4.0}  and \ref{thm:kuo}, and some perturbation arguments based on \eqref{appro:1.1} and
\eqref{appro:1.2}, we shall prove the following theorem 
\begin{thm}\label{thm:4}
Let $1 < q < \infty$ and $-1 + 1/q < s < 1/q$.  Let $\sigma > 0$ be a small number such that 
$-1+1/q < s-\sigma < s < s+\sigma < 1/q$. Assume that $s$ satisfies \eqref{assump:s}
and $\sigma$ \eqref{assump:sigma}. Let $\eta_0 = \gamma_* + \tilde\eta_0(x)$ with 
$\tilde\eta_0\in B^{s+1}_{q,1}(\HS)$ and 
let $\tilde\eta^\epsilon_0(x)$ be a 
function satisfying assumptions \eqref{appro:1.1} and \eqref{appro:1.2}. 
Set $\eta^\epsilon_0
= \gamma_* + \tilde\eta^\epsilon_0$. 
Then, there exist constants 
$\gamma > 0$ and $C>0$ depending on $\|\tilde\eta_0\|_{B^{N/q}_{q,1}(\HS)}$
such that for any $\lambda \in \Sigma_\mu+ \gamma$
 and $\bg \in B^s_{q,1}(\HS)$, 
problem \eqref{lame:1} admits a unique solution
 $\bv \in B^{s +2}_{q,1}(\HS)^N$
satisfying  the estimate:
\begin{equation}\label{est:1.1}
\|(\lambda, \lambda^{1/2}\bar\nabla, \bar\nabla^2)
\bv\|_{B^{s}_{q,1}(\HS)} 
\leq C\|\bg\|_{B^s_{q,1}(\HS)}
\end{equation}
for some constant $C > 0$. \par
Moreover,  there exist
constants $\gamma>0$ and $C>0$ depending on $\|\tilde\eta_0\|_{B^{N/q}_{q,1}(\HS)}$
 such that 
for any $\lambda \in \Sigma_\mu+\gamma$ 
there holds 
\begin{equation}
\|(\lambda, \lambda^{1/2}\bar\nabla, \bar\nabla^2)\bv\|_{B^s_{q,1}(\HS)}
\leq C|\lambda|^{-\frac{\sigma}{2}}\|\bg\|_{B^{s+\sigma}_{q,1}(\HS)},
\label{fundest.2*} \end{equation}
provided  $\bg \in B^{s+\sigma}_{q,1}(\HS)
 \cap B^s_{q,1}(\HS)$ as well as 
\begin{equation}
\|(\lambda, \lambda^{1/2}\bar\nabla, \bar\nabla^2)\pd_\lambda\bv\|_{B^s_{q,1}(\HS)}
\leq C|\lambda|^{-(1-\frac{\sigma}{2})} \|\bg\|_{B^{s-\sigma}_{q,1}(\HS)}.
\label{fundest.3*}
\end{equation}
provided $\bg \in B^{s-\sigma}_{q,1}(\HS) \cap B^s_{q,1}(\HS)$.
\end{thm}
\begin{remark} Since $B^{s+\sigma}_{q,1}(\HS) \subset B^s_{q,1}(\HS) \subset B^{s-\sigma}_{q,1}(\HS)$, 
to obtain \eqref{fundest.2*} and \eqref{fundest.3*} it suffices to assume that 
$\bg \in B^{s+\sigma}_{q,1}(\HS)$.
\end{remark}
\begin{proof}
Let $\nu = s$ or $s\pm\sigma$.  Note that $\tilde\eta_0 \in B^{s+1}_{q,1} \subset B^{N/q}_{q,1}$. 
For $\bg \in B^\nu_{q,1}$, let $\bv \in B^{\nu+2}_{q,1}$ be a unique 
solution of equations \eqref{lame:2}. By Theorem \ref{thm:4.0}, 
we know that $\bv$ satisfies the estimate:
\begin{equation}\label{july.21.1}
\|(\lambda, \lambda^{1/2}\bar\nabla, \bar\nabla^2)\bv\|_{B^\nu_{q,1}}
\leq C\|\bg\|_{B^\nu_{q,1}}.
\end{equation}
 Inserting $\bv$ into \eqref{lame:1}, we have
$$\eta^\epsilon_0(x)\lambda\bv - \alpha\Delta \bv- \beta\nabla\dv\bv = \bg + 
(\eta^\epsilon_0(x) - \eta_0(x))\lambda\bv, \quad\text{in $\BR^N_+$}, \quad \bv|_{\pd\HS} = 0.
$$
By Lemma \ref{lem:APH} and \eqref{july.21.1} we have 
\begin{align*}
\|(\eta^\epsilon_0- \eta_0)\lambda\bv\|_{B^\nu_{q,1}} &\leq C\|\tilde\eta^\epsilon_0-\tilde\eta_0\|_{B^{N/q}_{q,1}}
\|\lambda\bv\|_{B^\nu_{q,1}} \\
&\leq C\|\tilde\eta^\epsilon_0-\tilde\eta_0\|_{B^{s+1}_{q,1}}\|\bg\|_{B^\nu_{q,1}}.
\end{align*}
We choose $\epsilon>0$ so small that $C\|\tilde\eta^\epsilon_0-\tilde\eta_0\|_{B^{s+1}_{q,1}} \leq 1/2$.
If we define an operator $T_\nu$ by $\bv = T_\nu \bg$,  then 
$\|(\eta^\epsilon_0- \eta_0)\lambda T_\nu \bg\|_{B^\nu_{q,1}} \leq (1/2)\|\bg\|_{B^\nu_{q,1}}.$ Thus, 
the inverse map: $(\bI + (\eta^\epsilon_0- \eta_0)\lambda T_\nu)^{-1}$ exists as an operator of 
$\CL(B^\nu_{q,1})$ and $\|(\bI + (\eta^\epsilon_0- \eta_0)\lambda T_\nu)^{-1}\|_{\CL(B^\nu_{q,1})} \leq 2.$
Thus, defining an operator $U_\nu$ by $U_\nu= T_\nu(\bI + (\eta^\epsilon_0- \eta_0)\lambda T_\nu)^{-1}$,
we see that for any $\bg \in B^\nu_{q,1}$, $\bv = U_\nu\bg$ satisfies equations \eqref{lame:1}
and estimate:
\begin{equation}\label{july:21.2}
\|(\lambda, \lambda^{1/2}\bar\nabla, \bar\nabla^2)\bv\|_{B^\nu_{q,1}}
\leq C_\nu \|\bg\|_{B^\nu_{q,1}}
\end{equation}
for some constant $C_\nu$ depending on $\nu = s$ or $s\pm\sigma$.
This completes the proof of \eqref{est:1.1}. \par 

Let $\bv$ be a solution of equations \eqref{lame:1} for $\bg \in B^{s\pm\sigma}_{q,1} \cap B^s_{q,1}$.  We consider $\bv$ 
as a solution of  equations
$$\gamma_*\lambda\bv - \alpha\Delta\bv - \beta\nabla\dv\bv=
\bg - \tilde\eta^\epsilon_0(x)\lambda\bv \quad
\text{in $\HS$}, \quad \bv|_{\pd\HS}=0,
$$
where we have used the relation $\eta^\epsilon_0 = \gamma_* + \tilde\eta^\epsilon_0$. 
Applying  \eqref{fundest.2} and \eqref{fundest.3} of Theorem \ref{thm:kuo}, 
we have
\begin{align*}
\|(\lambda, \lambda^{1/2}\bar\nabla, \bar\nabla^2)\bv\|_{B^s_{q,1}}
&\leq C
|\lambda|^{-\frac{\sigma}{2}}(\|\bg\|_{B^{s+\sigma}_{q,1}}
+\|\tilde\eta^\epsilon_0 \lambda\bv\|_{B^{s+\sigma}_{q,1}}), \\
\|(\lambda, \lambda^{1/2}\bar\nabla, \bar\nabla^2)\pd_\lambda \bv\|_{B^s_{q,1}}
&\leq C
|\lambda|^{-(1-\frac{\sigma}{2})}(\|\bg\|_{B^{s-\sigma}_{q,1}}
+\|\tilde\eta^\epsilon_0 \lambda\bv\|_{B^{s-\sigma}_{q,1}}).
\end{align*}
Since 
$$\|\tilde\eta^\epsilon_0 \lambda\bv\|_{B^{s\pm\sigma}_{q,1}}
\leq C\|\tilde\eta^\epsilon_0\|_{B^{N/q}_{q,1}}\|\lambda\bv\|_{B^{s\pm\sigma}_{q,1}}
\leq  C\|\tilde\eta^\epsilon_0\|_{B^{N/q}_{q,1}}\|\lambda\bv\|_{B^{s\pm\sigma}_{q,1}},
$$
as follows from Lemme \ref{lem:APH}, by \eqref{july:21.2} and $\tilde\eta_0 \in B^{s+1}_{q,1} \subset B^{N/q}_{q,1}$,
we have
\begin{align*}
\|(\lambda, \lambda^{1/2}\bar\nabla, \bar\nabla^2)\bv\|_{B^s_{q,1}}
&\leq C(1+\|\tilde\eta_0\|_{B^{N/q}_{q,1}})|\lambda|^{-\frac{\sigma}{2}}
\|\bg\|_{B^{s+\sigma}_{q,1}}, \\
\|(\lambda, \lambda^{1/2}\bar\nabla, \bar\nabla^2)\pd_\lambda \bv\|_{B^s_{q,1}}
&\leq C(1+\|\tilde\eta_0\|_{B^{N/q}_{q,1}})|\lambda|^{-(1-\frac{\sigma}{2})}
\|\bg\|_{B^{s-\sigma}_{q,1}},
\end{align*}
which shows  \eqref{fundest.2*} and \eqref{fundest.3*}.
This completes the proof of Theorem \ref{thm:4}.
\end{proof}

Now, we consider  problem \eqref{s:2} of the Stokes system
and prove Theorem \ref{thm:3}.  
 We insert the relation: 
 $\rho = \lambda^{-1}(f-\eta^\epsilon_0\dv\bv)$ obtained 
from the first equation in \eqref{s:2} into the second equations.  Then, we have
\begin{equation}\label{SL:1}
\eta_0^\epsilon(x)\lambda\bv - \alpha\Delta \bv - \beta\nabla\dv\bv -
\lambda^{-1}\nabla(P'(\eta^\epsilon_0)\eta^\epsilon_0\dv\bv) = \bh
\quad\text{in $\HS$}, \quad \bu|_{\pd\HS} =0,
\end{equation}
where we have set $\bh = \bg - \lambda^{-1}\nabla(P'(\eta^\epsilon_0)f)$ for notational
simplicity. 
We shall show the following lemma.
\begin{lem}\label{lem:15} Let $1 < q < \infty$, $\sigma>0$ and  
$-1+1/q < s-\sigma < s < s+\sigma < 1/q$.
Assume that
$s$ satisfies \eqref{assump:s} and $\sigma$ satisfies
 \eqref{assump:sigma}.  Let $\nu=s$ or $s\pm \sigma$.  Assume that 
$\tilde\eta_0 \in B^{N/q}_{q,1}(\HS)$.
 Then, there exist constants $\gamma$ and $C$ such that 
for any $\lambda \in \Sigma_\mu + \gamma$ and $\bh \in B^\nu_{q,1}(\HS)$, 
problem \eqref{SL:1} admits a unique solution $\bv \in B^{\nu +2}_{q,1}(\HS)$ possessing the 
estimate:
\begin{equation}\label{23.6.20.7}
\|(\lambda, \lambda^{1/2}\bar\nabla, \bar\nabla^2)\bv\|_{B^\nu_{q,1}}
\leq C\|\bh\|_{B^\nu_{q,1}}
\end{equation}
for any $\lambda \in \Sigma_\mu + \gamma$. \par
Here, $\gamma$ depends on $\gamma_*$, $\|\tilde\eta_0\|_{B^{N/q}_{q,1}(\HS)}$,
and $\|\nabla\tilde\eta_0^\epsilon\|_{B^{N/q}_{q,1}(\HS)}$, and $C$ depends on 
$\gamma_*$ and $\|\tilde\eta_0\|_{B^{N/q}_{q,1}(\HS)}$.
\end{lem}
\begin{proof}
We shall solve equations \eqref{SL:1} by successive approximation for large $\lambda$. 
By Lemma \ref{lem:APH}, we have
$$
\|\nabla(P'(\eta_0^\epsilon)\eta_0^\epsilon\dv\bv)\|_{B^\nu_{q,1}}
\leq C(\|(P''(\eta_0^\epsilon)\eta_0^\epsilon + P'(\eta_0^\epsilon))
(\nabla\eta^\epsilon_0)\dv\bv\|_{B^\nu_{q,1}}
+ \|P'(\eta_0^\epsilon)\eta_0^\epsilon \nabla\dv\bv\|_{B^\nu_{q,1}}).
$$
We now use the following lemma for the Besov norm estimate of composite functions cf. 
\cite[Proposition 2.4]{H11} and \cite[Theorem 2.87]{BCD}.
\begin{lem}\label{lem:Hasp} Let $1 < q < \infty$. 
Let $I$ be an open interval of $\BR$.  Let $\omega>0$ and let $\tilde\omega$ be the smallest
integer such that $\tilde\omega \geq \omega$. Let $F:I \to \BR$ satisfy $F(0) = 0$ and 
$F' \in BC^{\tilde\omega}_\infty(I, \BR)$.  Assume that $v\in B^\omega_{q,r}$
has valued in $J \subset\subset I$.  Then, 
$F(v) \in B^\omega_{q,1}$ and there exists a constant $C$ depending only on 
$\nu$, $I$, $J$, and $N$, such that 
$$\|F(v)\|_{B^\omega_{q,1}} \leq C(1 + \|v\|_{L_\infty})^{\tilde\omega}\|F'\|_{BC^{\tilde\omega}_\infty(I,\BR)}
\|v\|_{B^\omega_{q,1}}.$$
\end{lem}
Recalling that $\eta^\epsilon_0 = \gamma_* + \tilde\eta^\epsilon_0$, we write
\begin{align*}
&(P''(\eta_0^\epsilon)\eta_0^\epsilon + P'(\eta_0^\epsilon)) \\
&= (P''(\gamma_*)+ \int^1_0P'''(\gamma_* + \ell\eta^\epsilon_0)\,\d\ell \eta^\epsilon_0)
 (\gamma_*+\eta^\epsilon_0) + P'(\gamma_*) + \int^1_0P''(\gamma_* + 
\ell\tilde\eta^\epsilon_0)\,\d\ell \tilde\eta^\epsilon_0\\
& = P''(\gamma_*)\gamma_*+ P'(\gamma_*) + Q_1(\tilde\eta^\epsilon_0)
\end{align*}
where we have set
\begin{align*}
Q_1(s) & = \gamma_*\int^1_0P'''(\gamma_*+\ell s)\,\d\ell s 
+ (P''(\gamma_*)+ \int^1_0P'''(\gamma_* + \ell s)\,\d\ell s)s
+ \int^1_0P''(\gamma_* + \ell s)\,\d\ell s.
\end{align*}
In view of \eqref{assump:0} and \eqref{appro:1.1}, we may assume that 
\begin{equation}\label{assump:2.1}
\rho_1/2 < \eta^\epsilon_0 < 2\rho_2,
\end{equation}
and so 
$$\rho_1/2-\gamma_* < \tilde\eta^\epsilon_0 < 2\rho_2-\gamma_*.$$
Thus, for $\ell \in (0, 1)$ we may assume that
\begin{equation}\label{assump:2.1}
\rho_1/2 -\gamma_*< \ell\tilde\eta^\epsilon_0 < 2\rho_2-\gamma_*
\end{equation}
for any $\ell \in (0, 1]$. From this observation, we may assume that $Q_1(s)$ is defined for 
$s \in (\rho_1/2-\gamma_*, 2\rho_2-\gamma_*)$ and $Q_1(0) = 0$.

By Lemmas \ref{lem:APH} and \ref{lem:Hasp}, and \eqref{appro:1.1}  we have 
\begin{align*}
&\|(P''(\eta_0^\epsilon)\eta_0^\epsilon+ P'(\eta_0^\epsilon))\nabla\tilde\eta^\epsilon_0
\dv\bv\|_{B^{\nu}_{q,1}}\\
&\quad \leq  C(|P''(\gamma_*)\gamma_*+P'(\gamma_*)|\|\nabla\eta^\epsilon_0\|_{B^{N/q}_{q,1}}
\|\dv\bv\|_{B^{\nu}_{q,1}}
+ \|Q_1(\tilde\eta^\epsilon_0)\|_{B^{N/q}_{q,1}}\|\nabla\eta^\epsilon_0\|_{B^{N/q}_{q,1}}
\|\dv\bv\|_{B^{\nu}_{q,1}} \\
&\quad \leq C(\gamma_*, \|\tilde\eta_0\|_{B^{N/q}_{q,1}}, 
 \|\nabla\tilde\eta^\epsilon_0\|_{B^{N/q}_{q,1}})\|\bv\|_{B^{\nu+2}_{q,1}}. 
\end{align*}
Here and in the sequel, $C(\gamma_*, \|\tilde\eta_0
\|_{B^{N/q}_{q,1}}, \|\nabla\tilde\eta^\epsilon_0\|_{B^{N/q}_{q,1}})$ denotes a constant depending on $\gamma_*$,  
$\|\tilde\eta_0\|_{B^{N/q}_{q,1}}$ and $\|\nabla\tilde\eta^\epsilon_0\|_{B^{N/q}_{q,1}}$.

Likewise, we write
$$P'(\eta^\epsilon_0)\eta^\epsilon_0 = (P'(\gamma_*) + \int^1_0 P''(\gamma_*+\ell\tilde\eta^\epsilon_0)\,
\d\ell \tilde\eta^\epsilon_0)(\gamma_* + \tilde\eta^\epsilon_0)
= P'(\gamma_*)\gamma_* + Q_2(\tilde\eta^\epsilon_0),
$$
where we have set
$$Q_2(s) = \int^1_0P''(\gamma_* + \ell s)\,\d\ell s \gamma_* + 
(P'(\gamma_*) + \int^1_0 P''(\gamma_* + \ell s)\,\d\ell s)s$$
for $s \in (\rho_1/2-\gamma_*, 2\rho_2-\gamma_*)$ and $Q_2(0) = 0$.
By Lemmas \ref{lem:APH} and  \ref{lem:Hasp}, we have
$$\|P'(\eta^\epsilon_0)\eta^\epsilon_0\nabla\dv\bv\|_{B^\nu_{q,1}}
\leq C(|P'(\gamma_*)\gamma_* | + (1+\|\tilde\eta^\epsilon_0\|_{L_\infty})^m\|\tilde\eta^\epsilon_0
\|_{B^{N/q}_{q,1}})\|\nabla\dv\bv\|_{B^\nu_{q,1}}.
$$
Therefore, we have
\begin{equation}\label{23.jn.20.3}
\|\nabla(P'(\eta_0^\epsilon)\eta_0^\epsilon\dv\bv)\|_{B^\nu_{q,1}} \leq 
C(\gamma_*, \|\tilde\eta_0
\|_{B^{N/q}_{q,1}}, \|\nabla\tilde\eta^\epsilon_0\|_{B^{N/q}_{q,1}})\|\bv\|_{B^{\nu+2}_{q,1}}.
\end{equation}
Choosing $\gamma>0$ so large that 
$\gamma^{-1}C(\gamma_*, \|\tilde\eta_0
\|_{B^{N/q}_{q,1}}, \|\nabla\tilde\eta^\epsilon_0\|_{B^{N/q}_{q,1}}) \leq 1/2$, we have
$$
|\lambda|^{-1}\|\nabla(P'(\eta^\epsilon_0)\eta^\epsilon_0\dv\bv)\|_{B^\nu_{q,1}}
\leq (1/2)\|\bv\|_{B^{\nu+2}_{q,1}}
$$
for any $\lambda \in \Sigma_\mu + \gamma$. 
\par
Thus, moving the term $\gamma^{-1}\nabla(P'(\eta^\epsilon_0)\eta^\epsilon_0\dv\bv)$ 
to the right hand side
in equations \eqref{SL:1} and using a successive  approximation method
based on Theorem \ref{thm:4}, 
we can prove that there exist positive constants  $\gamma$ depending on 
$\gamma_*$, $\|\tilde\eta_0\|_{B^{N/q}_{q,1}}$ and $\|\nabla\tilde\eta^\epsilon_0\|_{B^{N/q}_{q,1}}$,
 and $C$ depends on 
$\gamma_*$ and $\|\tilde\eta_0\|_{B^{N/q}_{q,1}}$ such that
problem \eqref{SL:1} admits a unique solution 
$\bv \in B^{\nu+2}_{q,1}$ satisfying the estimate \eqref{23.6.20.7}. 
This completes the proof of Lemma \ref{lem:15}.
\end{proof}
We now consider \eqref{SL:1} with $\bh = \bg-\lambda^{-1}\nabla(P'(\eta^\epsilon_0)f)$. 
In the same manner as in the proof of \eqref{23.jn.20.3}, we have
\begin{equation}\label{fundest.5}
\|\bh\|_{B^\nu_{q,1}} \leq C(\|\bg\|_{B^\nu_{q,1}} + |\lambda|^{-1}C(\gamma_*, \|\tilde\eta_0\|_{B^{N/q}_{q,1}}, 
\|\nabla\tilde\eta_0\|_{B^{N/q}_{q,1}})\|f\|_{B^{\nu+1}_{q,1}}).
\end{equation}
Choosing $\gamma$ so large that $\gamma^{-1}C(\gamma_*, \|\tilde\eta_0\|_{B^{N/q}_{q,1}}, 
\|\nabla\tilde\eta_0\|_{B^{N/q}{q,1}})  \leq 1$, and using  Lemma \ref{lem:15}, we see that 
problem \eqref{SL:1} admits a unique solution $\bv \in 
B^{\nu+2}_{q,1}$ satisfying the estimate:
\begin{equation}\label{finalest.1}
\|(\lambda, \lambda^{1/2}\bar\nabla, \bar\nabla^2)\bv\|_{B^{\nu}_{q,1}}
\leq C(\|f\|_{B^{\nu+1}_{q,1}} + \|\bg\|_{B^\nu_{q,1}})
\end{equation}
for $\nu = s$ and $s\pm \sigma$ and $\lambda \in \Sigma_\epsilon + \gamma$. 
Here and in the sequel, the constant $\gamma>0$ depends on $\gamma_*$, 
$\|\tilde\eta_0\|_{B^{N/q}_{q,1}}$,
and $\|\nabla\tilde\eta^\epsilon_0\|_{B^{N/q}_{q,1}}$, 
and  $C$ depends on $\gamma_*$ and $\|\tilde\eta_0\|_{B^{N/q}_{q,1}}$,
and we will not mention this fact in the sequel. 
\par 
Finally, define $\rho$ by  
$\rho=\lambda^{-1}(f-\eta^\epsilon_0\dv\bv)$.  Recall that $N/q \leq s+1$, and then by Lemma \ref{lem:APH}
and Lemma \ref{lem:Hasp}, and \eqref{finalest.1} with $\nu=s$, we have  
\begin{align*}
\|\lambda\rho\|_{B^{s+1}_{q,1}} &\leq C(\|f\|_{B^{s+1}_{q,1}} + (\gamma_*+\|\tilde\eta_0^\epsilon\|_{B^{s+1}_{q,1}})
\|\dv\bv\|_{B^{N/q}_{q,1}} +(\gamma_* +  \|\tilde\eta_0^\epsilon\|_{B^{N/q}_{q,1}})\|\nabla\dv\bv\|_{B^{s}_{q,1}})
\\
& \quad \leq  C(\gamma_*, \|\tilde\eta_0\|_{B^{s+1}_{q,1}})(\|f\|_{B^{s+1}_{q,1}}
+ \|\bg\|_{B^s_{q,1}})
\end{align*}
for every $\lambda \in \Sigma_\epsilon + \gamma$, because of  
$\tilde\eta_0 \in B^{s+1}_{q,1} \subset B^{N/q}_{q,1}$. 
This completes the proof of \eqref{est:1.0} in \thetag1 of Theorem \ref{thm:3}. 

We now prove \eqref{fundest.2**} and \eqref{fundest.3**}. Applying \eqref{fundest.2*} 
and \eqref{fundest.3*} to \eqref{SL:1}, we have
\begin{align*}
\|(\lambda, \lambda^{1/2}\bar\nabla, \bar\nabla^2)\bv\|_{B^s_{q,1}}& \leq C|\lambda|^{-\frac{\sigma}{2}}
(\|\bh\|_{B^{s+\sigma}_{q,1}} + |\lambda|^{-1}\|\nabla(P'(\eta_0^\epsilon)\eta^\epsilon_0\dv\bv)\|_{B^{s+\sigma}_{q,1}}),
\\
\|(\lambda, \lambda^{1/2}\bar\nabla, \bar\nabla^2)\pd_\lambda\bv\|_{B^s_{q,1}}& 
\leq C|\lambda|^{-(1-\frac{\sigma}{2})}
(\|\bh\|_{B^{s-\sigma}_{q,1}} + |\lambda|^{-1}\|\nabla(P'(\eta_0^\epsilon)\eta^\epsilon_0\dv\bv)\|_{B^{s-\sigma}_{q,1}}).
\end{align*}
for any $\lambda \in \Sigma_\mu + \gamma$ with $\bh= \bg-\lambda^{-1}\nabla(P'(\eta^\epsilon_0)f)$. 
By \eqref{23.jn.20.3}, we have 
$$|\lambda|^{-1}\|\nabla(P'(\eta_0^\epsilon)\eta_0^\epsilon\dv\bv)\|_{B^{s\pm\sigma}_{q,1}} \leq 
|\lambda|^{-1}C(\gamma_*, \|\tilde\eta_0
\|_{B^{N/q}_{q,1}}, \|\nabla\tilde\eta^\epsilon_0\|_{B^{N/q}_{q,1}})\|\bv\|_{B^{s+2+\pm\sigma}_{q,1}}.
$$
By \eqref{fundest.5}, we have
$$\|\bh\|_{B^\nu_{q,1}} \leq C(\|\bg\|_{B^\nu_{q,1}} + |\lambda|^{-1}C(\gamma_*, \|\tilde\eta_0\|_{B^{N/q}_{q,1}}, 
\|\nabla\tilde\eta_0\|_{B^{N/q}{q,1}})\|f\|_{B^{\nu+1}_{q,1}}).$$
Combining these estimates with Lemma \ref{lem:15} for
$\nu = s\pm\sigma$ and choosing $\gamma > 0$ so large that 
$$\gamma^{-1}C(\gamma_*, \|\tilde\eta_0\|_{B^{N/q}_{q,1}}, \|\nabla\tilde\eta^\epsilon_0\|_{B^{N/q}_{q,1}}) \leq 1,$$
we have
\begin{align*}
\|(\lambda, \lambda^{1/2}\bar\nabla, \bar\nabla^2)\bv\|_{B^s_{q,1}}& \leq C|\lambda|^{-\frac{\sigma}{2}}
(\|f\|_{B^{s+1+\sigma}_{q,1}} + \|\bg\|_{B^{s+\sigma}_{q,1}}), \\
\|(\lambda, \lambda^{1/2}\bar\nabla, \bar\nabla^2)\pd_\lambda\bv\|_{B^s_{q,1}}& 
\leq C|\lambda|^{-(1-\frac{\sigma}{2})}
(\|f\|_{B^{s+1-\sigma}_{q,1}} + \|\bg\|_{B^{s-\sigma}_{q,1}})
\end{align*}
for every $\lambda \in \Sigma_\mu + \gamma$. This shows \eqref{fundest.2**} and \eqref{fundest.3**}.

Finally, we shall prove \eqref{rho:1}.    Recalling that $\rho$ is defined by the formula:
$\lambda \rho= f-(\gamma_* + \tilde\eta_0^\epsilon)\dv\bv$, and 
using \eqref{finalest.1}, we have
\begin{equation}\label{diffest:1}\begin{aligned}
\|\rho\|_{B^{s+1}_{q,1}} &\leq C|\lambda|^{-1}(\|f\|_{B^{s+1}_{q,1}}
+ C(\gamma_*, \|\tilde\eta_0\|_{B^{N/q}_{q,1}}, \|\nabla\eta^\epsilon_0\|_{B^{N/q}_{q,1}})\|\bv\|_{B^{s+2}_{q,1}}) \\
&\leq |\lambda|^{-1}C(\gamma_*, \|\tilde\eta_0\|_{B^{N/q}_{q,1}}, \|\nabla\eta^\epsilon_0\|_{B^{N/q}_{q,1}})
(\|f\|_{B^{s+1}_{q,1}} + \|\bg\|_{B^s_{q,1}}).
\end{aligned}\end{equation}
Choosing $\gamma>0$ so large that $\gamma^{-(1-\frac{\sigma}{2})}
C(\gamma_*, \|\tilde\eta_0\|_{B^{N/q}_{q,1}}, \|\nabla\eta^\epsilon_0\|_{B^{N/q}_{q,1}}) \leq 1$, we have
\begin{equation}\label{diffest:2}
\|\rho\|_{B^{s+1}_{q,1}} \leq C|\lambda|^{-\frac{\sigma}{2}}
(\|f\|_{B^{s+1}_{q,1}}+ \|\bg\|_{B^s_{q,1}})
\end{equation}
for every $\lambda \in \Sigma_\mu + \gamma$. Thus, we have the 
first part of \eqref{rho:1}. \par
Differentiating the formula: $\lambda \rho= f-\eta_0\dv\bv$ with respect to 
$\lambda$, we have
$\pd_\lambda \rho = -\lambda^{-1}(\rho+\eta_0^\epsilon\dv\pd_\lambda\bv)$.  
By Lemma \ref{lem:APH}, we have
\begin{equation}\label{diffrho.1}
\|\pd_\lambda\rho\|_{B^{s+1}_{q,1}}  \leq |\lambda|^{-1}(\|\rho\|_{B^{s+1}_{q,1}}
+ C(\gamma_*, \|\tilde\eta_0\|_{B^{N/q}_{q,1}}, \|\nabla\tilde\eta^\epsilon_0\|_{B^{N/q}_{q,1}})
\|\pd_\lambda\bv\|_{B^{s+2}_{q,1}}).
\end{equation}
To estimate $\pd_\lambda\bv$, we differentiate \eqref{s:2},  which reads  
$$\left\{\begin{aligned}
\lambda \pd_\lambda\rho+\eta_0^\epsilon\dv\pd_\lambda\bv = -\rho& 
&\quad&\text{in $\HS$}, \\
\eta_0^\epsilon\lambda\pd_\lambda\bv - \alpha\Delta \pd_\lambda\bv 
 -\beta\nabla\dv\pd_\lambda\bv 
+  \nabla(P'(\eta_0^\epsilon) \pd_\lambda\rho) = -\eta_0^\epsilon\bv& &\quad&\text{in $\HS$}, \\
\pd_\lambda\bv|_{\pd\HS} =0.
\end{aligned}\right.$$
Applying \eqref{finalest.1} with $\nu=s$ implies 
$$
\|(\lambda, \lambda^{1/2}\bar\nabla, \bar\nabla^2)\pd_\lambda\bv\|_{B^s_{q,1}}
\leq C(\|\rho\|_{B^{s+1}_{q,1}} +(\gamma_*+\|\tilde\eta^\epsilon_0\|_{B^{N/q}_{q,1}}) \|\bv\|_{B^s_{q,1}}).
$$
We use \eqref{finalest.1} with $\nu=s$ to obtain 
$$\|\bv\|_{B^s_{q,1}} = |\lambda|^{-1}\|\lambda \bv\|_{B^s_{q,1}}
\leq C|\lambda|^{-1}(\|f\|_{B^{s+1}_{q,1}} + \|\bg\|_{B^s_{q,1}}).$$
Combining these estimates and using \eqref{diffest:1} implies 
\begin{align*}
\|\pd_\lambda\rho\|_{B^{s+1}_{q,1}}
&\leq C(\gamma_*, \|\tilde\eta_0\|_{B^{N/q}_{q,1}}, \|\nabla\tilde\eta^\epsilon_0
\|_{B^{N/q}_{q,1}})|\lambda|^{-2}(\|f\|_{B^{s+1}_{q,1}} + \|\bg\|_{B^s_{q,1}}).
\end{align*}
Choosing $\gamma>0$ so large that $\gamma^{-1-\frac{\sigma}{2}}
C(\gamma_*, \|\tilde\eta_0\|_{B^{N/q}_{q,1}}, \|\nabla\tilde\eta^\epsilon_0\|_{B^{N/q}_{q,1}}) \leq 1$, 
we see that 
$$\|\pd_\lambda \rho\|_{B^{s+1}_{q,1}} \leq |\lambda|^{-(1-\frac{\sigma}{2})}
(\|f\|_{B^{s+1}_{q,1}} + \|\bg\|_{B^s_{q,1}})$$
for every $\lambda \in \Sigma_\mu + \gamma$. 
Thus, we have proved \eqref{rho:1}. 
 This completes the proof of Theorem \ref{thm:3}. 
\section{$L_1$ semigroup}

In this section, we assume that $1 < q < \infty$, $\sigma>0$, $-1/q < s-\sigma < s < s+\sigma < 1/q$,
and that $s$ satisfies \eqref{assump:s} and $\sigma$ satisfies \eqref{assump:sigma}.  
Let $\eta_0(x) = \gamma_* + \tilde\eta_0(x)$ and we assume that $\tilde\eta_0(x) \in B^{s+1}_{q,1}(\HS)$
and satisfy the conditions \eqref{assump:0}. Let $\tilde\eta^\epsilon_0(x)$ be an regularization of $\tilde\eta_0(x)$
satisfying \eqref{appro:1.1} and 
\eqref{appro:1.2}, and set $\rho^\epsilon_0(x) = \gamma_* + \tilde\rho^\epsilon_0(x)$. 
  From \eqref{assump:s}  we know that $N/q \leq s+1$,
and so $\tilde\eta_0 \in B^{N/q}_{q,1}(\HS)$. 
In the sequel, $\mu$ is a fixed constant such that $0 < \mu < \pi/2$. \par 
In this section, we consider evolution equations:
\begin{equation}\label{semi:1}\left\{\begin{aligned}
\pd_t\rho + \eta_0^\epsilon(x)\dv\bu &=F&\quad&\text{in $\HS\times(0, T)$}, \\
\eta_0^\epsilon(x)\pd_t\bu - \alpha\Delta\bu - \beta\nabla\dv\bu + \nabla(P'(\eta^\epsilon_0)\rho)
& = \bG &\quad&\text{in $\HS\times(0, T)$}, \\
\bu|_{\pd\HS} = 0, \quad (\rho, \bu) = (f, \bg)&&\quad
&\text{in $\HS$}.
\end{aligned}\right.\end{equation}
The corresponding resolvent problem to  \eqref{semi:1} reads 
equations \eqref{s:2}.  Let 
$$\CH = \{(f, \bg) \mid f \in B^{s+1}_{q,1}(\HS), \enskip
\bg \in B^s_{q,1}(\HS)^N\} = B^{s+1}_{q,1}(\HS)\times B^s_{q,1}(\HS)^N.
$$
Let $\CA$ and $\CD(\CA)$ be an operator and 
its domain corresponding to  equations \eqref{semi:1} defined by
\begin{align*}
\CD(\CA) &= \{(\rho, \bu) \in \CH \mid \bu \in B^{s+2}_{q,1}(\HS), 
\quad \bu|_{\pd\HS} = 0\}, \\
\CA(\rho, \bu) &= (\eta^\epsilon_0\dv\bu, \,\,
-\alpha\eta^\epsilon_0(x)^{-1}\Delta\bu -
\beta\eta^\epsilon_0(x)\nabla\dv\bu+ \eta^\epsilon_0(x)^{-1}\nabla(P'(\eta^\epsilon_0(x)\rho))
\end{align*}
Then, problem \eqref{s:2}  reads 
\begin{equation}\label{resol:0}
(\lambda\bI + \CA)(\rho, \bu)= (f, \bg).
\end{equation}

The following theorem follows from  Theorem \ref{thm:3}. 
\begin{thm}\label{thm:t.1}
Let $1 < q < \infty$ and $-1 + 1/q \leq s < 1/q$.  Assume that $s$ satisfies the condition
\eqref{assump:s}. Let $\eta_0(x)$ be a 
function given in Theorem \ref{thm:1}. Then, an operator $\CA$
generates a $C_0$ analytic semigroup on $\CH$. 
\end{thm}
\begin{proof} Noting that $-1+N/q \leq s < 1/q$ as follows from \eqref{assump:s}, by Theorem \ref{thm:3} 
we know that there exist two constants $\gamma$ and $C$ such that for any $\lambda \in 
\Sigma_\mu + \gamma$, $f \in B^{s+1}_{q,1}(\HS)$ and $\bg \in B^s_{q,1}(\HS)^N$ problem \eqref{s:2} 
admits a unique solution $\rho \in B^{s+1}_{q,1}(\HS)$ and $\bv \in B^{s+2}_{q,1}(\HS)^N$ satisfying the 
estimates:
\begin{equation}\label{resol:3.1}\|(\lambda, \lambda^{1/2}\bar\nabla, \bar\nabla^2)\bv\|_{B^s_{q,1}(\HS)}
+ \|\lambda\rho\|_{B^{s+1}_{q,1}(\HS)} 
 \leq C(\|f\|_{B^{s+1}_{q,1}(\HS)} + \|\bg\|_{B^s_{q,1}(\HS)}). 
\end{equation}
We see that  $\gamma$ depends on $\gamma_*$, $\|\tilde\eta_0\|_{B^{N/q}_{q,1}(\HS)}$
 and $\|\nabla\tilde\eta^\epsilon_0\|_{B^{N/q}_{q,1}(\HS)}$ and  that $C$ depends on  $\gamma_*$ 
and $\|\tilde\eta_0\|_{B^{s+1}_{q,1}(\HS)}$, because $\|\tilde\eta_0\|_{B^{N/q}_{q,1}(\HS)} \leq C\|\tilde\eta_0\|_{B^{s+1}_{q,1}
(\HS)}$. 
In particular, from \eqref{resol:3.1} it follows that $(\lambda \bI+\CA)^{-1}$
exists for $\lambda \in \Sigma_\epsilon + \gamma$ and 
\begin{equation}\label{resol:3.2}
\|\lambda(\lambda\bI + \CA)(f, \bg)\|_\CH + \|(\lambda\bI+\CA)^{-1}(f, \bg)\|_{\CD(\CA)}
\leq C\|(f, \bg)\|_{\CH}.
\end{equation}
Thus, by holomorphic semigroup theory (cf. \cite{YK}), we see the generation of $C^0$ analytic semigroup
 associated with equations \eqref{semi:1}. 

\end{proof}
Let $\{T(t)\}_{t\geq 0}$ be a $C_0$ analytic semigroup generated by
$\CA$ and we shall prove its $L_1$ maximal regularity. 
The idea is due to  Kuo \cite{Kuo23}, and also due to Shibata and Watanabe
\cite{ SW1, SW2}. Below, we write 
$$T(t)(f, \bg) = (T_1(t)(f, \bg), T_2(t)(f, \bg)).$$
Let $\rho = T_1(t)(f, \bg)$ and $\bu = T_2(t)(f, \bg)$, and then  $\rho$ and $
\bu$ satisfy equations \eqref{semi:1} with $F=0$ and $\bG=0$. 
We shall prove the following theorem.
\begin{thm}\label{thm:t.2}
Let $1 < q < \infty$ and $-1+1/q < s < 1/q$. Assume that $s$ satisfies the condition \eqref{assump:s}.  
Let $\eta_0(x)$ be a 
function given in Theorem \ref{thm:1}. 
Let $\{T(t)\}_{t\geq 0}$ be a continuous analytic semigroup
generated by $\CA$.  Then, there exist positive constants $\gamma$ and $C$ such that for any
$(f, \bg) \in \CH$, there holds
$$\int^\infty_0e^{-\gamma t}(\|(\pd_t, \bar\nabla^2)T_2(t)(f, \bg)\|_{B^s_{q,1}(\HS)}
+ \|(1, \pd_t)T_1(t)(f, \bg)\|_{B^{s+1}_{q,1}(\HS)})\,\d t
\leq C\|(f, \bg)\|_{\CH}.
$$
Here, $\gamma$ depends on $\gamma_*$, 
$\|\eta^\epsilon_0\|_{B^{N/q}_{q,1}(\HS)}$
 and $\|\nabla\tilde\eta^\epsilon_0\|_{B^{N/q}_{q,1}(\HS)}$,  and $C$ depends on 
$\gamma_*$ and $\|\eta_0\|_{B^{N/q}_{q,1}(\HS)}$.
\end{thm}
\begin{proof}
Let $(\theta, \bv) = (\lambda+\CA)^{-1}(f, \bg)$, then $\theta \in 
B^{s+1}_{q,1}(\HS)$ and $\bv \in B^{s+2}_{q,1}(\HS)^N$ satisfy 
equations \eqref{s:2}.  Since $B^{s+1+\sigma}_{q,1}(\HS)\times B^{s+\sigma}_{q,1}(\HS)^N$
is dense in $B^{s+1}_{q,1}(\HS)\times B^s_{q,1}(\HS)^N$, we may assume that 
$(f, \bg) \in B^{s+1+\sigma}_{q,1}(\HS)\times B^{s+\sigma}_{q,1}(\HS)^N$  below.
Thus, by Theorem \ref{thm:3} we know that
\begin{align}\label{6.21.1}
\|(\lambda, \bar\nabla^2)\bv\|_{B^{s}_{q,1}} &\leq C|\lambda|^{-\frac{\sigma}{2}}
\|(f, \bg)\|_{B^{s+1+\sigma}_{q,1}(\HS)\times B^{s+\sigma}_{q,1}(\HS)}, 
\\
\label{6.21.2}
\|(\lambda, \bar\nabla^2)\pd_\lambda\bv
\|_{B^s_{q,1}(\HS)} &\leq C|\lambda|^{-(1-\frac{\sigma}{2})}
\|(f, \bg)\|_{B^{s+1-\sigma}_{q,1}(\HS)\times B^{s-\sigma}_{q,1}(\HS)}
\end{align}
for every $\lambda \in \Sigma_\mu +\gamma$. 
Here, $\gamma$ depends on $\gamma_*$, 
$\|\eta^\epsilon_0\|_{B^{N/q}_{q,1}(\HS)}$
 and $\|\nabla\tilde\eta^\epsilon_0\|_{B^{N/q}_{q,1}(\HS)}$,  and $C$ depends on 
$\gamma_*$ and $\|\eta_0\|_{B^{N/q}_{q,1}(\HS)}$.
\par
Let $\Gamma = \Gamma_+ \cup \Gamma_-$ be a contour in the complex plane $\BC$ defined by
$$\Gamma_\pm = \{\lambda = re^{i(\pi\pm\epsilon)} \mid r \in (0, \infty)\}.$$
Here, $\epsilon \in (0, \pi/2)$.  According to  well-known  Holomorphic 
semigroup theory  (cf. \cite[p.257]{YK}), $T(t)$ is represented by
$$T(t)(f, \bg) = \frac{1}{2\pi i}\int_{\Gamma + \gamma} e^{\lambda t}
(\lambda\bI+\CA)^{-1}(f, \bg) \,\d\lambda \quad\text{for $t>0$}.
$$
Notice that $(\lambda\bI+\CA)^{-1}(f, \bg) = (\theta, \bv)$. 
We have  
$$T_1(t)(f, \bg) =\frac{1}{2\pi i}\int_{\Gamma + \gamma} e^{\lambda t}
\theta \,\d\lambda,
\quad  T_2(t)(f, \bg) = \frac{1}{2\pi i}\int_{\Gamma + \gamma} e^{\lambda t}
\bv \,\d\lambda.$$
Let $\CH_\pm = B^{s+1\pm\sigma}_{q,1}(\HS)\times B^{s\pm\sigma}_{q,1}(\HS)$.
By change of variable: $\lambda t= \ell$ and  by \eqref{6.21.1} and \eqref{6.21.2}, we have
\begin{equation}\label{6.21.3}\begin{aligned}
\|\bar\nabla^2 T_2(t)(f, \bg)\|_{B^s_{q,1}(\HS)}
&\leq Ce^{\gamma t}t^{-1+\frac{\sigma}{2}}\|(f, \bg)\|_{\CH_{+\sigma}}, \\
\|\bar\nabla^2 T_2(t)(f, \bg)\|_{B^s_{q,1}(\HS)}
&\leq Ce^{\gamma t}t^{-1-\frac{\sigma}{2}}\|(f, \bg)\|_{\CH_{-\sigma}}.
\end{aligned}\end{equation}
In fact, noting that ${\rm Re}\, e^{\lambda t} = e^{t(\gamma + r\cos(\pi\pm\epsilon)}
= e^{\gamma t}e^{-tr\cos\epsilon}$ for $\lambda 
\in \Gamma_\pm + \gamma$, by \eqref{6.21.1} we have
\begin{align*}
\|\bar\nabla^2T_2(t)(f, \bg)\|_{B^s_{q,1}(\HS)}
& \leq Ce^{\gamma t}\int^\infty_0e^{-tr\cos\epsilon}
\|\bar\nabla^2\bv\|_{B^s_{q,1}(\HS)}\, \d r \\ 
& \leq Ce^{\gamma t}\int^\infty_0e^{-tr\cos\epsilon}r^{-\frac{\sigma}{2}}
\, \d r \, \|(f, \bg)\|_{\CH_{+\sigma}}
\\
& = Ce^{\gamma t}t^{-1+\frac{\sigma}{2}}
\int^\infty_0e^{-s\cos\epsilon}s^{-\frac{\sigma}{2}}
\, \d s \, \|(f, \bg)\|_{\CH_{+\sigma}}.
\end{align*}
Thus, we have the first inequality in \eqref{6.21.3}. 
To prove the second inequality in \eqref{6.21.3}, we write
$$\bar\nabla^2T_2(t)(f, \bg)
= -\frac{1}{2\pi i t}\int_{\Gamma+\gamma}
e^{\lambda t}\pd_\lambda (\bar\nabla^2\bv)\,d\lambda.
$$
And then, by \eqref{6.21.1}
\begin{align*}
\|\bar\nabla^2T_2(t)(f, \bg)\|_{B^s_{q,1}(\HS)}
& \leq Ct^{-1}e^{\gamma t}\int^\infty_0e^{-tr\cos\epsilon}
\|\bar\nabla^2\pd_\lambda \bv\|_{B^s_{q,1}(\HS)}\, \d r \\ 
& \leq Ct^{-1}e^{\gamma t}\int^\infty_0e^{-tr\cos\epsilon}r^{-1+\frac{\sigma}{2}}
\, \d r \, \|(f, \bg)\|_{\CH_{-\sigma}}
\\
& = Ce^{\gamma t}t^{-1-\frac{\sigma}{2}}
\int^\infty_0e^{-s\cos\epsilon}s^{-1+\frac{\sigma}{2}}
\, \d s \, \|(f, \bg)\|_{\CH_{-\sigma}}. 
\end{align*}
Thus, we have the second inequality of \eqref{6.21.3}.
\par
Since
$$ T_1(t)(f, \bg) = \frac{1}{2\pi i}
\int_{\Gamma+\gamma} e^{\lambda t} \theta\,\d\lambda,$$
by \thetag3 of Theorem \ref{thm:3}, we also have
\begin{equation}\label{6.21.4}\begin{aligned}
\|T_1(t)(f, \bg)\|_{B^{s+1}_{q,1}(\HS)}
&\leq Ce^{-\gamma t}t^{-1+\frac{\sigma}{2}}\|(f, \bg)\|_{\CH_{+\sigma}}, \\
\|T_1(t)(f, \bg)\|_{B^{s+1}_{q,1}(\HS)}
&\leq Ce^{-\gamma t}t^{-1-\frac{\sigma}{2}}\|(f, \bg)\|_{\CH_{-\sigma}}.
\end{aligned}\end{equation}
Using \eqref{6.21.3} and \eqref{6.21.4} and real interpolation method, we have
\begin{align*}
\int^\infty_0 e^{-\gamma t}\|\bar\nabla^2T_2(t)(f, \bg)\|_{B^{s}_{q,1}(\HS)}\,\d t &\leq C\|(f, \bg)\|_\CH, \\
\int^\infty_0 e^{-\gamma t}\|T_1(t)(f, \bg)\|_{B^{s+1}_{q,1}(\HS)}\,dt
&\leq C\|(f, \bg)\|_{\CH}.
\end{align*}
In fact, we write 
\begin{align*}
&\int^\infty_0 e^{-\gamma t}\|\bar\nabla^2T_2(t)(f, \bg)\|_{B^s_{q,1}(\HS)}\,\d t \\
&\quad = \sum_{j \in \BZ} \int^{2^{(j+1)}}_{2^j}e^{-\gamma t}
\|\bar\nabla^2T_2(t)(f, \bg)\|_{B^s_{q,1}(\HS)}\, \d t\\
& \quad\leq \sum_{j \in \BZ} \int^{2^{(j+1)}}_{2^j}
\sup_{t \in (2^j, 2^{j+1})}(e^{-\gamma t}
\|\bar\nabla^2T_2(t)(f, \bg)\|_{B^s_{q,1}(\HS)})\,\d t
\\
& \quad =2\sum_{j \in \BZ} 2^j
\sup_{t \in (2^j, 2^{j+1})}(e^{-\gamma t}\|\bar\nabla^2T_2(t)(f, \bg)\|_{B^s_{q,1}(\HS)}).
\end{align*}
Setting $a_j = \sup_{t \in (2^j, 2^{j+1})}e^{-\gamma t}
\|\bar\nabla^2T_2(t)(f, \bg)\|_{B^s_{q,1}(\HS)}$, 
we have
$$
\int^\infty_0 e^{-\gamma t}\|\bar\nabla^2T_2(t)(f, \bg)\|_{B^s_{q,1}(\HS)}\,\d t
\leq 2((2^ja_j))_{\ell_1} = 2((a_j)_{j \in \BZ})_{\ell^1_1}. 
$$
Here and in the following, 
$\ell^s_q$ denotes the set of all sequences $(2^{js}a_j)_{j \in \BZ}$ such that 
\begin{align*}
\|((a_j)_{j \in \BZ})\|_{\ell^s_q} &= \Bigl\{\sum_{j \in \BZ} 
(2^{js}|a_j|)^q \Bigr\}^{1/q} < \infty \quad 1 \leq q < \infty, \\
\|((a_j)_{j \in \BZ})\|_{\ell^s_\infty} &= \sup_{j \in \BZ} 
2^{js}|a_j| < \infty \quad  q = \infty.
\end{align*}
By \eqref{6.21.3}, we have
$$\sup_{j \in \BZ} 2^{j(1-\frac{\sigma}{2})}a_j \leq C\|(f, \bg)\|_{\CH_{+\sigma}}, 
\quad\sup_{j \in \BZ} 2^{j(1+\frac{\sigma}{2})}a_j \leq C\|(f, \bg)\|_{\CH_{-\sigma}}.
$$
Namely, we have
$$\|(a_j)\|_{\ell_\infty^{1-\frac{\sigma}{2}}} \leq C\|(f, \bg)\|_{\CH_{+\sigma}},
\quad \|(a_j)\|_{\ell_\infty^{1+ \frac{\sigma}{2}}} \leq C\|(f, \bg)\|_{\CH_{-\sigma}}.
$$
According to \cite[5.6.1.Theorem]{BL},  
we know that $\ell^1_1 = (\ell^{1-\frac{\sigma}{2}}_\infty,
\ell^{1+\frac{\sigma}{2}}_\infty)_{1/2, 1}$, 
where
$(\cdot, \cdot)_{\theta, q}$ denotes the real interpolation functor, and therefore we have 
\begin{equation}\label{L1est:1}
\int^\infty_0 e^{-\gamma t}\|\bar\nabla^2T_2(t)(f, \bg)\|_{B^s_{q,1}(\HS)}\,\d t
\leq C\|(f, \bg)\|_{(\CH_{+\sigma}, \CH_{-\sigma})_{1/2, 1}}
\leq C\|(f, \bg)\|_{B^{s+1}_{q,1}(\HS)\times B^s_{q,1}(\HS)}.
\end{equation}
Employing completely the same argument, by \eqref{6.21.4} we have
\begin{equation}\label{L1est:2}\begin{aligned}
\int^\infty_0 e^{-\gamma t}\|T_1(t)(f, \bg)\|_{B^{s+1}_{q,1}(\HS)}\,\d t
&\leq C\|(f, \bg)\|_{(\CH_{+\sigma}, \CH_{-\sigma})_{1/2, 1}}
\leq C\|(f, \bg)\|_{B^{s+1}_{q,1}(\HS)\times B^s_{q,1}(\HS)}.
\end{aligned}\end{equation}
\par

By equations \eqref{semi:1} with $f=0$ and $\bg=0$, we have
$$\pd_t \rho = -\eta_0^\epsilon(x)\dv\bu, \quad
\pd_t\bu =(\eta^\epsilon_0)^{-1}( \alpha\Delta\bu + \beta\nabla\dv\bu - \nabla(P'(\eta^\epsilon_0)\rho))
$$
 with  $\rho=T_1(t)(f, \bg)$ and 
$\bu = T_2(t)(f, \bg)$. Note that $\rho_1/2-\gamma_* < \tilde\eta^\epsilon_0(x) < \rho_2-\gamma_*$ as follows from
$\eta^\epsilon_0 = \gamma_* + \tilde\eta^\epsilon_0$, 
\eqref{assump:0} and \eqref{appro:1.1}.  By \eqref{L1est:1} and \eqref{L1est:2},  we have 
\begin{align*}
\int^\infty_0 e^{-\gamma t}
\|\pd_tT_1(t)(f,\bg)\|_{B^{s+1}_{q,1}(\HS)}\,\d t
&\leq C(\gamma_*+\|\tilde\eta_0\|_{B^{s+1}_{q,1}(\HS)})
\int^\infty_0 e^{-\gamma t}
\|\dv T_2(t)(f, \bg)\|_{B^{s+1}_{q,1}(\HS)}\,\d t \\
& \leq C(\gamma_*+\|\tilde\eta_0\|_{B^{s+1}_{q,1}(\HS)})
\|(f, \bg)\|_{\CH}, \\
\int^\infty_0 e^{-\gamma t}\|\pd_t T_2(t)(f, \bg)\|_{B^{s}_{q,1}(\HS)}\,\d t
& \leq  C(\gamma_*, \|\tilde\eta_0\|_{B^{N/q}_{q,1}})
\Bigl\{\int^\infty_0\|\nabla^2 T_2(t)(f, \bg)\|_{B^s_{q,1}(\HS)} \\
&\quad 
+ C(\gamma_*, \|\tilde\eta_0\|_{B^{s+1}_{q,1}(\HS)})
\int^\infty_0 \|T_1(t)(f, \bg)\|_{B^{s+1}_{q,1}(\HS)}\,\d t\Bigr\} \\
& \leq C(\gamma_*, \|\tilde\eta_0\|_{B^{s+1}_{q,1}(\HS)})
\|(f, \bg)\|_{\CH}.
\end{align*}
This completes the proof of Theorem \ref{thm:t.2}. \end{proof}
\begin{cor}\label{semi.1}
Let $1 < q < \infty$ and $T > 0$. Let $s$ be a number satisfying \eqref{assump:s}. 
  Let $\eta_0(x)=\gamma_* + \tilde\eta_0(x)$ be a 
function given in Theorem \ref{thm:1}.  Then, for any 
$(f, \bg) \in \CH$, $F\in L_1((0, T),  B^{s+1}_{q,1}(\HS))$ and 
$\bG \in L_1((0, T), B^s_{q,1}(\HS)^N)$, problem \eqref{semi:1} admits unique
solutions $\rho$ and $\bu$ with
$$\rho \in W^1_1((0, T), B^{s+1}_{q,1}(\HS)), \quad
\bu \in L_1((0, T), B^{s+2}_{q,1}(\HS)^N) \cap W^1_1((0, T), B^s_{q,1}(\HS)^N).
$$
Moreover,  there exist constants  $\gamma>0$ depending on $\gamma_*$, 
 $\|\tilde\eta_0\|_{B^{N/q}_{q,1}(\HS)}$ and $\|\nabla\tilde\eta^\epsilon_0\|_{B^{N/q}_{q,1}}$,
and 
$C$ depending on $\gamma_*$ and $\|\tilde\eta_0\|_{B^{s+1}_{q,1}(\HS)}$
such that $\rho$ and $\bu$ satisfy the following maximal
$L_1$-$\CH$ estimate:
\begin{align*}
\|(\pd_t, \bar\nabla^2)\bu\|_{L_1((0, T), B^s_{q,1}(\HS))}
+ \|(1, \pd_t)\rho\|_{L_1((0, T), B^{s+1}_{q,1}(\HS))} 
\leq Ce^{\gamma T}(\|(f, \bg)\|_{\CH}
+ \|(F, \bG)\|_{L_1((0, T), \CH)}).
\end{align*}
\end{cor}
\begin{proof} 
Let $F_0$ and $\bG_0$ be zero extension of $F$ and $\bG$ outside of $(0, T)$
interval.  Using $\{T(t)\}_{t\geq 0}$, we can write
$$(\rho, \bu) = T(t)(f, \bg) + \int^t_0T(t-s)(F_0, \bG_0)(s)\,ds.
$$
Let $\gamma$ and $C$ be the constant given in Theorem \ref{thm:t.2}. 
By Fubini's theorem, we have
\begin{align*}
&\int^\infty_0e^{-\gamma t}\|\bar\nabla^2
\int^t_0T_2(t-\ell)(F_0, \bG_0)(\ell)\,\d\ell\|_{B^s_{q,1}(\HS)}\, \d t \\
&\quad \leq \int^\infty_0\Bigl\{ \int^\infty_\ell e^{-\gamma t}
\|\bar\nabla^2T(t-\ell)(F_0, \bG_0)(\ell)\|_{B^s_{q,1}(\HS)}\,\d t\Bigr\}\,\d \ell \\
&\quad = \int^\infty_0e^{-\gamma \ell}
\Bigl\{ \int^\infty_0 e^{-\gamma t}
\|\bar\nabla^2 T(t)(F_0, \bG_0)(\ell)\|_{B^{s}_{q,1}(\HS)}\,\d t\Bigr\}\,\d \ell \\
&\quad \leq C\int^\infty_0 e^{-\gamma \ell} \|(F_0(\cdot, \ell), \bG_0(\cdot, \ell)\|_{\CH}\,\d \ell \\
&\quad \leq C\|(F, \bG)\|_{L_1((0, T), \CH)}.
\end{align*}
Completely the same argument, we have
$$\int^\infty_0e^{-\gamma t}\|
\int^t_0T_1(t-\ell)(F_0, \bG_0)(\ell)\,\d \ell\|_{B^{s+1}_{q,1}(\HS)}\, \d t 
\leq C\|(F, \bG)\|_{L_1((0, T), \CH)}.
$$
Therefore, we have
$$\int^\infty_0 e^{-\gamma t}(\|\rho(\cdot, t)\|_{B^{s+1}_{q,1}(\HS)}
+ \|\bu(\cdot, t)\|_{B^{s+2}_{q,1}(\HS)})\,dt
\leq C(\|(f, \bg)\|_{\CH} + \|(F, \bG)\|_{L_1((0, T), \CH)}), 
$$
which implies that
$$e^{-\gamma T}\int^T_0(\|\rho(\cdot, t)\|_{B^{s+1}_{q,1}(\HS)}
+ \|\bu(\cdot, t)\|_{B^{s+2}_{q,1}(\HS)})\,dt
\leq C(\|(f, \bg)\|_{\CH} + \|(F,  \bG)\|_{L_1((0, T), \CH)}).
$$
Therefore, we have
$$\int^T_0(\|\rho(\cdot, t)\|_{B^{s+1}_{q,1}(\HS)}
+ \|\bu(\cdot, t)\|_{B^{s+2}_{q,1}(\HS)})\,dt
\leq Ce^{\gamma T}(\|(f, \bg)\|_{\CH} + \|(F, \bG)\|_{L_1((0, T), \CH)}).
$$
Here, $\gamma$ is a constant depending on $\gamma_*$, 
$\|\tilde\eta_0\|_{B^{N/q}_{q,1}(\HS)}$ and $\|\tilde\eta^\epsilon_0\|_{B^{N/q}_{q,1}(\HS)}$,
and $C$ is a constant depending on $\gamma_*$, and 
$\|\tilde\eta_0\|_{B^{N/q}_{q,1}(\HS)}$. \par 
To show the estimate of time derivatives, we use the relations:
\begin{align*}
\pd_t\rho& = -\eta_0^\epsilon\dv\bu + F, \\
\pd_t\bu& =(\eta^\epsilon_0)^{-1}(\alpha\Delta\bu + \beta\nabla\dv\bu - 
\nabla(P(\eta_0^\epsilon)\rho)+ \bG),
\end{align*}
and then, 
\begin{align*}
&\int^T_0(\|\pd_t\rho(\cdot, t)\|_{B^{s}_{q,1}(\HS)}
+ \|\pd_t\bu(\cdot, t)\|_{B^{s}_{q,1}(\HS)})\,dt \\
&\leq C(\gamma_*, \|\tilde\eta_0\|_{B^{s+1}_{q,1}})(
\int^T_0(\|\rho(\cdot, t)\|_{B^{s+1}_{q,1}(\HS)}
+ \|\bar\nabla^2\bu(\cdot, t)\|_{B^{s}_{q,1}(\HS)})\,dt
+ \|(F, \bG)\|_{L_1((0, T), \CH)}) \\
& \leq C(\gamma_*, \|\tilde\eta_0\|_{B^{s+1}_{q,1}})e^{\gamma T}
(\|(f, \bg)\|_{\CH} + \|(F, \bG)\|_{L_1((0, T), \CH)}).
\end{align*}
Noting that $N/q \leq s+1$, we see that $C$ depends on $\gamma_*$ and 
$\|\tilde\eta_0\|_{B^{s+1}_{q,1}(\HS)}$.  Thus, we have obtained  Corollary \ref{semi.1}.

\end{proof}


\section{A proof of Theorem \ref{thm:2}}

In this section, we shall prove Theorem \ref{thm:2}. 
In what follows, we assume that $\theta_0 \in B^{s+1}_{q,1}(\HS)$ 
and $\bu_0 \in B^s_{q,1}(\HS)^N$, which satisfy the compatibility
condition: $\bu_0|_{\pd\HS} =0$. Let $\tilde\eta^\epsilon_0$ be an element 
of $ 
B^{s+1}_{q,1}(\HS) \cap B^{N/q+1}_{q,1}(\HS)$ such that 
\begin{equation}\label{appro:1} 
\lim_{\epsilon\to0}
\|\tilde\eta^\epsilon_0- \tilde \eta_0\|_{B^{s+1}_{q,1}(\HS)}= 0.
\end{equation}
We divide equations \eqref{ns:2} into linear parts and nonlinear parts 
by setting $\rho = \theta_0 + \theta$.  Moreover, we write $\rho = \theta_0 + \theta
= \eta^\epsilon_0 + \theta_0-\eta^\epsilon_0 + \theta$.
The resultan equations read
\begin{equation}\label{appro:ns.1}\left\{\begin{aligned}
\pd_t\theta+\eta^\epsilon_0\dv\bu = (\eta^\epsilon_0-\theta_0-\theta)\dv\bu 
 + F(\theta+\theta_0, \bu)& 
&\quad&\text{in $\HS\times(0, T)$}, \\
\eta^\epsilon_0 \pd_t\bu - \alpha\Delta \bu  -\beta\nabla\dv\bu 
+  \nabla (P'(\eta^\epsilon_0)\theta)
 =-\nabla P(\theta_0)   
 + \bG(\theta+\theta_0, \bu) + \tilde\bG(\theta, \bu) && \quad&\text{in $\HS\times(0, T)$}, \\
\bu|_{\pd\HS} =0, \quad (\theta, \bu)|_{t=0} = (0, \bu_0)&
& \quad&\text{in $\HS$},
\end{aligned}\right.\end{equation}
where we have set $\tilde\bG(\theta, \bu) 
 = (\eta^\epsilon_0 - \theta_0)\pd_t\bu  - \nabla( P(\theta_0+\theta)-P(\theta_0) 
-P'(\eta^\epsilon_0)\theta)$. 

To prove Theorem \ref{thm:2}, we use  the Banach contraction
mapping principle. To this end, we introduce an energy functional
$E_T$ and an underlying space $S_{T, \omega}$ defined by 
\begin{align*}
E_T&(\eta, \bw)  = \|(\eta, \pd_t\eta)\|_{L_1((0, T), B^{s+1}_{q,1}(\HS))}
+ \|\bw\|_{L_1((0, T), B^{s+2}_{q,1}(\HS))}
+ \|\pd_t\bw\|_{L_1((0, T), B^s_{q,1}(\HS))}, \\[0.5pc]
S_{T, \omega}  &= \left\{ 
(\eta, \bw) \, \left | 
\begin{aligned} \eta &\in W^1_1((0, T), B^{s+1}_{q,1}(\HS)), \enskip 
\bw \in L_1((0, T), B^{s+2}_{q,1}(\HS)^N)
\cap W^1_1((0, T), B^s_{q,1}(\HS)^N) \\
(\eta, &\bw)|_{t=0} = (0, \bu_0), \quad E_T(\eta, \bw) \leq \omega,
\quad \int^T_0\|\nabla \bw(\cdot, \tau)\|_{B^{N/q}_{q,1}(\HS)}\,\d\tau \leq c_1
\end{aligned}
\right.\right\}.
\end{align*}
Here, $T>0$, $\omega>0$ and $c_1>0$  are small constants 
chosen later. \par
Given $(\theta, \bu) \in S_{T, \omega}$, let $\eta$ and $\bw$ be solutions
to the system of linear equations:
\begin{equation}\label{st:2}\left\{\begin{aligned}
\pd_t\eta+\eta^\epsilon_0\dv\bw =
(\eta^\epsilon_0 - \theta_0-\theta)\dv\bu + F(\theta_0+\theta, \bu)& 
&\quad&\text{in $\HS\times(0, T)$}, \\
\eta_0^\epsilon\pd_t\bw - \alpha\Delta \bw  -\beta\nabla\dv\bw
+  \nabla(P'(\eta^\epsilon_0) \eta) = -\nabla P(\theta_0) + \bG(\theta_0+\theta, \bu)
+\tilde\bG(\theta, \bu)
& &\quad&\text{in $\HS\times(0, T)$}, \\
\bw|_{\pd\HS} =0, \quad (\eta, \bw)|_{t=0} = (0, \bu_0)
& &\quad&\text{in $\HS$}.
\end{aligned}\right.\end{equation}
Let $\eta_\ba$ and $\bw_\ba$ be solutions of the system of linear equations:
\begin{equation}\label{st:3}\left\{\begin{aligned}
\pd_t\eta_\ba+\eta^\epsilon_0\dv\bw_\ba =0&
&\quad&\text{in $\HS\times(0, \infty)$}, \\
\eta^\epsilon_0\pd_t\bw_\ba - \alpha\Delta \bw_\ba   -\beta\nabla\dv\bw_\ba
+  \nabla(P'(\eta^\epsilon_0) \eta_\ba) = -\nabla P(\theta_0) 
& &\quad&\text{in $\HS\times(0, \infty)$}, \\
\bw_\ba|_{\pd\HS} =0, \quad (\eta_\ba, \bw_\ba)|_{t=0} = (0, \bu_0)
& &\quad&\text{in $\HS$}.
\end{aligned}\right.\end{equation}
We will choose $T>0$ small enough later, and so for a while we assume that
$0 < T < 1$. 
By Corollary \ref{semi.1}, we know the unique existence 
of solutions $\eta_\ba$ and $\bw_\ba$ satisfying the regularity conditions:
$$\eta_\ba \in W^1_1((0, 1), B^{s+1}_{q,1}(\HS)), 
\quad 
\bw_\ba \in L_1((0, 1), B^{s+2}_{q,1}(\HS)^N)
\cap W^1_1((0, 1), B^s_{q,1}(\HS)^N)
$$ 
as well as the estimates:
\begin{equation}\label{est:2}\begin{aligned}
&\|(\eta_\ba, \pd_t\eta_\ba)\|_{L_1((0, 1), B^{s+1}_{q,1}(\HS))}
+ \|(\pd_t, \bar\nabla^2)\bw_\ba\|_{L_1((0, 1), B^s_{q,1}(\HS))} 
\\
&\quad \leq Ce^{\gamma}(\|\bu_0\|_{B^s_{q,1}(\HS)} + \|\nabla P(\theta_0)\|_{B^s_{q,1}(\HS)}).
\end{aligned}\end{equation}
Here, $\gamma$ is a constant depending on $\gamma_*$, 
 $\|\tilde\eta_0\|_{B^{s+1}_{q,1}(\HS)}$, and $\|\nabla\tilde\eta^\epsilon_0\|_{B^{N/q}_{q,1}
(\HS)}$ 
given in Corollary \ref{semi.1}. Here and in the following, $C$ denotes a  general constant
depending at most on $\gamma_*$ and $\|\tilde\eta_0\|_{B^{s+1}_{q,1}(\HS)}$, 
which is changed from line to line,  
but independent of $\epsilon$.   \par
In view of \eqref{est:2},  $\eta_\ba$ and $\bw_\ba$ satisfy $E_1(\eta_\ba, \bw_\ba) < \infty$, 
 and so we choose $T \in (0, 1)$ small enough in such a way that
\begin{equation}\label{est:5}
E_T(\eta_\ba, \bw_\ba) \leq \omega/2.
\end{equation}
Let $\rho$ and $\bv$ be solutions to the system of linear equations:
\begin{equation}\label{st:40}\left\{\begin{aligned}
\pd_t\rho+\eta^\epsilon_0\dv\bv =(\eta^\epsilon_0-\theta_0)\dv\bu  -\theta\dv\bu +
F(\theta+\theta_0, \bu)& 
&\quad&\text{in $\HS\times(0, T)$}, \\
\eta^\epsilon_0\pd_t\bv- \alpha\Delta \bv-\beta\nabla\dv\bv
+  \nabla(P'(\eta^\epsilon_0) \rho) = \bG(\theta+\theta_0, \bu) + \tilde\bG(\theta, \bu)
& &\quad&\text{in $\HS\times(0, T)$}, \\
\bv|_{\pd\HS} =0, \quad (\rho, \bv)|_{t=0} = (0, 0)
& &\quad&\text{in $\HS$}.
\end{aligned}\right.\end{equation}
Applying Corollary \ref{semi.1}, 
we see the existence of solutions $\rho$ and $\bv$ of equations
\eqref{st:40} satisfying the regularity condition:
$$\rho \in W^1_1((0, T), B^{s+1}_{q,1}(\HS)), 
\quad
\bv \in L_1((0, T), B^{s+2}_{q,1}(\HS)^N) \cap W^1_1((0, T), B^{s}_{q,1}(\HS)^N)
$$
as well as the estimate:
\begin{equation}\label{est:3}\begin{aligned}
E_T(\rho, \bv) \leq Ce^{\gamma T}(&\|(\eta^\epsilon_0-\theta_0)\dv\bu, \theta\dv\bu,
F(\theta+\theta_0, \bu)\|_{L_1((0, T), B^{s+1}_{q,1}(\HS))} \\
&+ \|(\bG(\theta+\theta_0, \bu), \tilde\bG(\theta, \bu)) 
\|_{L_1((0, T), B^s_{q,1}(\HS))}).
\end{aligned}\end{equation}
Here, we notice that $\gamma$ depends on $\epsilon$ but $C$ is independent of $\epsilon$
 again. 
\par

Now, we  shall show that there exist  constants $C>0$ and $\epsilon >0$ such that 
\begin{equation}\label{est:4}\begin{aligned}
\|(\eta^\epsilon_0-\theta_0)\dv\bu, \theta\dv\bu,
F(\theta+\theta_0, \bu)\|_{L_1((0, T), B^{s+1}_{q,1}(\HS))}
&+ \|(\bG(\theta+\theta_0, \bu), \tilde\bG(\theta, \bu)) 
\|_{L_1((0, T), B^s_{q,1}(\HS))}  \\
&\leq C(\omega^2 + \omega^3).
\end{aligned}\end{equation}
If we show \eqref{est:4}, then by \eqref{est:3} we have
\begin{equation}\label{est:7}
E_T(\rho, \bv) \leq Ce^{\gamma T}(\omega^2 + \omega^3).
\end{equation}
Choose $\epsilon > 0$ and $\omega>0$  so small that $Ce(\omega + \omega^2) \leq 1/2$, 
and so $\gamma$ is fixed.  Next, we choose $T>0$ so small that $\gamma T\leq 1$.  Then, 
we have
\begin{equation}\label{est:6}
E_T(\rho, \bv) < \omega/2,
\end{equation}
which, combined with \eqref{est:5}, implies that 
$\eta = \eta_\ba +  \rho$ and $\bw =\bw_\ba + \bv$ satisfy
equations \eqref{st:2} and $E_T(\eta, \bw) < \omega$.
Especially, $\omega$ is chosen so small that 
$$\int^T_0\|\nabla\bw(\cdot, \tau)\|_{B^{N/q}_{q,1}(\HS)}\,\d\tau 
\leq CE_T(\eta, \bw) \leq C\omega \leq c_1.$$
As a consequence, $(\eta, \bw) \in S_{T, \omega}$. 
Thus,  if we define the map $\Phi$ by $\Phi(\theta, \bu) = (\eta, \bw)$,
then  $\Phi$ maps $S_{T, \omega}$ into $S_{T, \omega}$. \par

Now, we shall show \eqref{est:4}.  For notational simplicity, we omit $\HS$ below. 
Notice that $B^{N/q}_{q,1}$ is a Banach algebra (cf. \cite[Proposition 2.3]{H11}). 
By Lemma \ref{lem:APH} and 
the assumption: $N/q\leq  s+1$, we see that $B^{s+1}_{q,1}$ is also a Banach
algebra.  In fact, 
$$\|uv\|_{B^{s+1}_{q,1}} \leq \|(\nabla u)v\|_{B^s_{q,1}} + \|u\bar\nabla v\|_{B^s_{q,1}}
\leq C(\|\nabla u\|_{B^s_{q,1}}\|v\|_{B^{N/q}_{q,1}} +
\|u\|_{B^{N/q}_{q,1}}\|\bar\nabla v\|_{B^s_{q,1}})
\leq C\|u\|_{B^{s+1}_{q,1}}\|v\|_{B^{s+1}_{q,1}}.
$$
\par
We first estimate $(\eta_0^\epsilon-\theta_0-\theta)\dv\bu$ and 
$F(\theta+\theta_0, \bu)$.  Write $(\eta_0^\epsilon-\theta_0)
= \tilde\eta_0^\epsilon -\tilde\eta_0 + \eta_0-\theta_0$  
and choose $\epsilon >0$ and $\sigma>0$ so small that 
\begin{equation}\label{smalldist:1}
\|\eta^\epsilon_0 - \eta_0\|_{B^{s+1}_{q,1}} \leq \omega, \quad 
\|\eta_0 - \theta_0\|_{B^{s+1}_{q,1}} \leq \omega.
\end{equation}
Then, by Lemma \ref{lem:APH}, we have
\begin{equation}\label{nonest:1}
\|(\eta_0^\epsilon-\theta_0)\dv\bu\|_{B^{s+1}_{q,1}} \leq C\omega\|\bu\|_{B^{s+2}_{q,1}}.
\end{equation}
Since $B^{s+1}_{q,1}$ is a Banach algebral, we have
$$\|\theta\dv\bu\|_{B^{s+1}_{q,1}} \leq C\|\theta\|_{B^{s+1}_{q,1}}\|\dv\bu\|_{B^{s+1}_{q,1}}.$$
Since $\theta|_{t=0} = 0$, we observe that 
\begin{equation}\label{theta:1}
\|\theta\|_{B^{s+1}_{q,1}} \leq C\|\pd_t\theta\|_{L_1((0, T), B^{s+1}_{q,1})}.
\end{equation}
Thus, we have
$$\|\theta\dv\bu\|_{L_1((0, T), B^{s+1}_{q,1})} \leq C \|\pd_t\theta\|_{L_1((0, T), B^{s+1}_{q,1})}
\|\bu\|_{L_1((0, T), B^{s+2}_{q,1})}.
$$
We next estimate $F(\theta_0+\theta, \bu) 
= (\theta_0+\theta)((\BI - \BA_\bu):\nabla\bu)$. 
Recall that $\bu$ satisfies
\begin{equation}\label{defin:1}
\int^T_0\|\nabla\bu(\cdot, \tau)\|_{B^{N/\beta}_{q,1}}\, \d\tau \leq c_1.
\end{equation}
Since 
$B^{N/q}_{q,1} \subset L_\infty$, we have
\begin{equation}\label{defin:2}
\sup_{t \in (0, T)}\,\Bigl\|\int^t_0 \nabla \bu(\cdot, \tau)\,\d\tau\Bigr\|_{L_\infty}
\leq C\int^T_0\|\nabla \bu(\cdot, \tau)\|_{B^{N/\beta}_{q,1}}\,\d\tau 
\leq Cc_1.
\end{equation}
Choosing $c_1$ so small that $Cc_1 < 1$. 
Let $F(\ell)$ be a $C^\infty$ function defined on $|\ell| \leq Cc_1$ and 
$F(0)=0$, and 
$\BI-\BA_\bu = F(\int^t_0\nabla\bu\,\d\ell)$. In fact, 
$F(\ell) = -\sum_{j=1}^\infty \ell^j$.
Then, by Lemma \ref{lem:Hasp} and \eqref{defin:2}, we have
\begin{equation}\label{nonfun:1}
\sup_{t \in (0, T)} \|F(\int^t_0\nabla\bu\,\d\tau)\|_{B^{s+1}_{q,1}}
\leq C\int^T_0\|\nabla\bu(\cdot, \tau)\|_{B^{s+1}_{q,1}}\,d\tau.
\end{equation}
  Since $B^{s+1}_{q,1}$ is a Banach algebra, using \eqref{nonfun:1} we have
\begin{align*}
\|F(\theta_0+\theta, \bu)\|_{B^{s+1}_{q,1}}
\leq C(\|\theta_0\|_{B^{s+1}_{q,1}}+\|\theta(\cdot, t)\|_{B^{s+1}_{q,1}})
\|\bu\|_{L_1((0, T), B^{s+2}_{q,1})}\|\nabla\bu(\cdot, t)\|_{B^{s+1}_{q,1}}.
\end{align*}
Using \eqref{theta:1}, we have
$$\|F(\theta_0+\theta, \bu)\|_{L_1((0, T), B^{s+1}_{q,1})}
\leq C(\|\theta_0\|_{B^{s+1}_{q,1}}+\|\theta_t\|_{L_1((0, T), B^{s+1}_{q,1})})
\|\bu\|_{L_1((0, T), B^{s+2}_{q,1})}^2.
$$
Summing up, we have proved that 
\begin{equation}\label{mainest:1}\begin{aligned}
&\|(\eta^\epsilon_0-\theta_0)\dv\bu, \theta\dv\bu,
F(\theta+\theta_0, \bu)\|_{L_1((0, T), B^{s+1}_{q,1}(\HS))} \\
&\quad \leq C\{\omega \|\bu\|_{L_1((0, T), B^{s+2}_{q,1})}
+ \|\pd_t\theta\|_{L_1((0, T), B^{s+1}_{q,1})}\|\bu\|_{L_1((0, T), B^{s+2}_{q,1})} \\
&\quad +(\|\eta_0\|_{B^{s+1}_{q,1}}+1)\|\bu\|_{L_1((0, T), B^{s+2}_{q,1})}^2
+ \|\pd_t\theta\|_{L_1((0, T), B^{s+1}_{q,1})}\|\bu\|_{L_1((0, T), B^{s+2}_{q,1})}^2\}.
\end{aligned}\end{equation}
Here and in the following, we use the estimate: 
$$\|\theta_0\|_{B^{s+1}_{q,1}} \leq \|\theta_0-\eta_0\|_{B^{s+1}_{q,1}} + \|\eta_0\|_{B^{s+1}_{q,1}})
\leq 1+ \|\eta_0\|_{B^{s+1}_{q,1}}.
$$

Next, we estimate $\|(\bG(\theta+\theta_0, \bu), \tilde\bG(\theta, \bu)) 
\|_{L_1((0, T), B^s_{q,1}(\HS))}$. 
By Lemma \ref{lem:APH}, the assumption: $N/d \leq s+1$,  \eqref{theta:1}, and 
\eqref{nonfun:1}, 
we have
\begin{align*}
\|(\BI-\BA_\bu)(\theta_0+\theta)\pd_t\bu\|_{B^s_{q,1}}
&\leq C\|\BI-\BA_\bu\|_{B^{N/q}_{q,1}}\|\theta_0+\theta\|_{B^{N/q}_{q,1}}
\|\pd_t\bu\|_{B^s_{q,1}} \\
& \leq C\|\bu\|_{L_1((0, T), B^{s+2}_{q,1})}(\|\theta_0\|_{B^{s+1}_{q,1}}
+ \|\pd_t\theta\|_{L_1((0, T), B^{s+1}_{q,1}})\|\pd_t\bu\|_{B^s_{q,1}}, \\
\|(\BA_u^{-1}-\BI)\dv(\BA_\bu\BA_\bu^\top:\nabla\bu)\|_{B^s_{q,1}}
&\leq C\|\BA_u^{-1}-\BI\|_{B^{N/q}_{q,1}}(\|\dv \nabla\bu\|_{B^{s}_{q,1}}
+ \|(\BA_\bu\BA_\bu^\top-\BI):\nabla\bu\|_{B^{s+1}_{q,1}})\\
& \leq C(\|\bu\|_{L_1((0, T), B^{s+2}_{q,1})}(1
+ \|\bu\|_{L_1((0, T), B^{s+2}_{q,1})})\|\bu\|_{B^{s+2}_{q,1}}).
\end{align*}
Therefore, we have 
\begin{equation}\label{mainest:2}\begin{aligned}
&\|(\bG(\theta+\theta_0, \bu)\|_{L_1((0, T), B^s_{q,1})} \\
&\quad \leq 
C( \|\bu\|_{L_1((0, T), B^{s+2}_{q,1})}(\|\theta_0\|_{B^{s+1}_{q,1}}
+ \|\pd_t\theta\|_{L_1((0, T), B^{s+1}_{q,1}})\|\pd_t\bu\|_{L_1((0, T), B^s_{q,1})}\\
&\quad +
\|\bu\|_{L_1((0, T), B^{s+2}_{q,1})}(1
+ \|\bu\|_{L_1((0, T), B^{s+2}_{q,1})})\|\bu\|_{L_1((0, T), B^{s+2}_{q,1}}).
\end{aligned}\end{equation}

Next, we shall estimate 
$\tilde\bG(\theta, \bu) 
 = (\eta^\epsilon_0 - \theta_0)\pd_t\bu  + \nabla( P(\theta_0+\theta)-P(\theta_0) 
-P'(\eta^\epsilon_0)\theta)$. Using Lemma \ref{lem:APH} and  \eqref{smalldist:1},   we have
\begin{align*}
\|(\eta^\epsilon_0 - \theta_0)\pd_t\bu\|_{B^s_{q,1}}
&\leq C\|\eta^\epsilon_0 -\theta_0\|_{B^{N/q}_{q,1}}
\|\pd_t\bu\|_{B^s_{q,1}} \leq C\omega\|\pd_t\bu\|_{B^s_{q,1}}.
\end{align*}
To estimate the second term, we write 
\begin{align*}
&P(\theta_0+\theta) - P(\theta_0) - P'(\eta_0^\epsilon)\theta \\
&\quad = P(\theta_0+\theta)-P(\theta_0) -P'(\theta_0)\theta 
+ (P'(\theta_0)-P'(\eta_0^\epsilon))\theta\\
&\quad =  \int^1_0(1-\ell)P''(\theta_0+\ell\theta)\,\d\ell\theta^2 
+ \int^1_0P''(\eta_0^\epsilon + \ell(\theta_0-\eta^\epsilon_0))\,\d\ell (\theta_0-\eta^\epsilon_0)\theta.
\end{align*}
Write $\theta_0+\ell\theta = \eta_0 + \theta_0-\eta_0 + \ell\theta$. By \eqref{smalldist:1}
and $E_T(\theta, \bu) < \omega$,  we see that
$$\|\theta_0-\eta_0 + \ell\theta\|_{L_\infty} 
\leq C\|\theta_0-\eta_0\|_{B^{s+1}_{q,1}} + \ell\|\pd_t\theta\|_{L_1((0, T), B^{s+1}_{q,1})}
\leq C\omega
$$
for $\ell \in (0, 1)$. 
In view of \eqref{assump:0}, we choose  $\omega$ so small that
$$
\rho_1/2 < \eta_0 + \theta_0-\eta_0 +\ell\theta\leq 2\rho_2
$$
for any $\ell \in (0, 1)$. Recalling that $\eta_0 = \gamma_* + \tilde\eta_0$, we have
\begin{equation}\label{range:1}
\rho_1/2-\gamma_* <  \tilde\eta_0 + \theta_0-\eta_0+\ell\theta \leq 2\rho_2 -\gamma_*
\end{equation}
for any $\ell \in (0, 1)$.  From this  observation, we write
\begin{align*}
&\int^1_0(1-\ell)P''(\theta_0+\ell\theta)\,\d\ell\theta^2 \\
&= \int^1_0\int^1_0(1-\ell)P'''(\gamma_* + m(\tilde \eta_0 + \theta_0-\eta_0 +\ell\theta)
(\tilde\eta_0+\theta_0-\eta_0 + \ell\theta)\,\d m\d\ell\, \theta^2
+ \frac12 P''(\gamma_*)\theta^2.
\end{align*}

 And also, we write $\eta^\epsilon_0 + \ell(\theta_0-
\eta^\epsilon_0) = \eta_0 + (1-\ell)(\eta^\epsilon_0 -\eta_0) + \ell(\theta_0-
\eta_0)$ and  observe that 
\begin{equation}\label{domain:4.1}
\|\eta^\epsilon_0-\eta_0 + \ell(\theta_0-
\eta^\epsilon_0)\|_{L_\infty}
\leq C((1-\ell)\| \eta^\epsilon_0-\eta_0 \|_{B^{s+1}_{q,1}} + \ell\|\theta_0-\eta_0\|_{B^{s+1}_{q,1}})
\end{equation}
for any $\ell \in (0, 1)$, In view of \eqref{smalldist:1}, we choose
$\omega$ so small that 
$$\rho_1/2 < \eta^\epsilon_0 + \ell(\theta_0-\eta^\epsilon_0) < 2\rho_2$$
for any $\ell \in (0, 1)$ as follows from Assumption \eqref{assump:0}, we have
\begin{equation}\label{range:2}
\rho_1/2-\gamma_* <  \tilde\eta^\epsilon_0 + \ell(\theta_0-\eta^\epsilon_0) < 2\rho_2-\gamma_*
\end{equation}
for any $\ell \in (0, 1)$.  From this observation, we write 
\begin{align*}
&\int^1_0P''(\eta_0^\epsilon + \ell(\theta_0-\eta^\epsilon_0))\,\d\ell (\theta_0-\eta^\epsilon_0)\theta
\\
& = \int^1_0\int^1_0 P'''(\gamma_* + m(\tilde\eta_0^\epsilon
+  \ell(\theta_0-\eta^\epsilon_0))(\tilde\eta_0^\epsilon
+  \ell(\theta_0-\eta^\epsilon_0)
 \,\d\ell \d m (\theta_0-\eta^\epsilon_0)\theta
+ P''(\gamma_*)(\theta_0-\eta^\epsilon_0)\theta.
\end{align*}
Therefore, by Lemmas \ref{lem:APH} and \ref{lem:Hasp}, we have
\begin{align*}
\|\nabla( P(\theta_0+\theta)-P(\theta_0) 
-P'(\eta^\epsilon_0)\theta)\|_{B^s_{q,1}}
\leq C(\gamma_*, \|\tilde\eta_0\|_{B^{s+1}_{q,1}})( \|\theta\|_{B^{s+1}_{q,1}}^2 +\|\theta_0-\eta^\epsilon_0\|_{B^{s+1}_{q,1}}
\|\theta\|_{B^{s+1}_{q,1}}).
\end{align*}


Putting these estimates together and using \eqref{smalldist:1} and \eqref{theta:1}, we have
\begin{equation}\label{mainest:3}\begin{aligned}
\|\tilde\bG(\theta, \bu)\|_{L_1((0, T), B^s_{q,1})} 
&\leq C(\gamma_*, \|\tilde\eta_0\|_{B^{s+1}_{q,1}})\{
\omega(\|\pd_t\bu\|_{L_1((0, T), B^s_{q,1})}+\|\theta\|_{L_1((0, T), B^{s+1}_{q,1})}) \\
&\quad + \|\pd_t\theta\|_{L_1((0, T), B^{s+1}_{q,1})}\|\theta\|_{L_1((0, T), B^{s+1}_{q,1})}\}.
\end{aligned}\end{equation}

Combining \eqref{mainest:1}, \eqref{mainest:2}, \eqref{mainest:3}
and recalling that $E_T(\theta, \bu) \leq \omega$, we have \eqref{est:4}.
And so, first  choosing $\epsilon > $0, $\sigma>0$, and $\omega > 0$  so small that
 $Ce(\omega+\omega^2) \leq 1/2$, and then choosing $T>0$ small enough 
to control the largeness 
of $\|\nabla\tilde\eta^\epsilon_0\|_{B^{N/q}_{q,1}}$,
 that is  $\gamma T\leq 1$,
we have  \eqref{est:6}. 
Here, $C$ depends on $\gamma_*$ and $\|\tilde\eta_0\|_{B^{s+1}_{q,1}}$, 
and so the smallness of $\omega$, $\sigma>0$,  and $\epsilon > 0$ depends on $\gamma_*$ and 
$\|\tilde\eta_0\|_{B^{s+1}_{q,1}}$, and  the smallness of $T>0$ depends on $\gamma_*$, 
$\|\tilde\eta_0\|_{B^{s+1}_{q,1}(\HS)}$ and $\|\nabla\eta^\epsilon_0\|_{B^{N/q}_{q,1}}$
after choosing $\epsilon > 0$, $\sigma>0$, and $\omega>0$. 
 Therefore, we see that 
$\Phi$ maps $S_{T, \omega}$ into itself.  \par

We now prove that $\Phi$ is contractive.  To this end, 
pick up two elements $(\theta_i, \bu_i) \in S_{T, \omega}$
($i=1,2$), and let $(\eta_i, \bw_i) = \Phi(\theta_i, \bu_i) \in S_{T, \omega}$ be solutions of 
equations \eqref{st:2} with $(\theta, \bu) = (\theta_i, \bu_i)$. 
Let  
\begin{align*}
\Theta &= \eta_1-\eta_2, \quad 
\bU = \bw_1-\bw_2, \\
\BF&=(\eta^\epsilon_0-\theta_0)\dv(\bu_1-\bu_2)
-(\theta_1\dv\bu_1-\theta_2\dv\bu_2) + F(\theta_1+\theta_0, \bu_1) - F(\theta_2+\theta_0, \bu_2), \\
\BG &= \bG(\theta_1+\theta_0, \bu_1) - \bG(\theta_2+\theta_0, \bu_2) 
+(\eta^\epsilon_0 -\theta_0)\pd_t(\bu_1-\bu_2)\\
&+ \nabla(P(\theta_0+\theta_1)-P(\theta_0+\theta_2))
-\nabla(P'(\eta^\epsilon_0)(\theta_1-\theta_2)).
\end{align*}
Notice that   $\Theta$ and $\bU$ satisfy equations:
\begin{equation}\label{diff:1}\left\{\begin{aligned}
\pd_t\Theta+\gamma\dv\bU =\BF& 
&\quad&\text{in $\HS\times(0, T)$}, \\
\pd_t\bU - \alpha\Delta \bU  -\beta\nabla\dv\bU
+  \gamma\nabla \Theta = \BG
& &\quad&\text{in $\HS\times(0, T)$}, \\
\bU|_{\pd\HS} =0, \quad (\Theta, \bU)|_{t=0} = (0, 0)
& &\quad&\text{in $\HS$}.
\end{aligned}\right.\end{equation}
From \eqref{est:3}, it follows that 
\begin{equation}\label{diff:2}
E_T(\eta_1-\eta_2, \bw_1-\bw_2) \leq Ce^{\gamma T}(\|\BF\|_{L_1((0, T), B^{s+1}_{q,1})}
+ \|\BG\|_{L_1((0, T), B^{s+2}_{q,1})}).
\end{equation}
We shall prove that 
\begin{equation}\label{diff:3}
\|\BF\|_{L_1((0, T), B^{s+1}_{q,1})}
+ \|\BG\|_{L_1((0, T), B^{s+2}_{q,1})}
\leq C(\omega+\omega^2) E_T(\theta_1-\theta_2, \bu_1-\bu_2).
\end{equation}
We start with estimating $\BF$.  Recall that $B^{N/q}_{q,1}$ and 
$B^{s+1}_{q,1}$ are Banach algebra. 
By Lemma \ref{lem:APH} and \eqref{smalldist:1}
\begin{align*}\|(\eta^\epsilon_0-\theta_0)\dv(\bu_1-\bu_2)\|_{L_1((0, T), B^{s+1}_{q,1})}
&\leq C\|\eta^\epsilon_0-\theta_0\|_{B^{s+1}_{q,1}}\|\bu_1-\bu_2\|_{L_1((0, T), B^{s+2}_{q,1})}
\\
&\leq C\omega \|\bu_1-\bu_2\|_{L_1((0, T), B^{s+2}_{q,1})}.
\end{align*}
Writing $\theta_1\dv\bu_1-\theta_2\dv\bu_2= (\theta_1-\theta_2)\dv\bu_1+ \theta_2(\dv\bu_1-\dv\bu_2)$
and using Lemma \ref{lem:APH} and \eqref{theta:1} gives 
\begin{align*}
&\|\theta_1\dv\bu_1-\theta_2\dv\bu_2\|_{B^{s+1}_{q,1}} \leq C(\|\dv\bu_1\|_{B^{s+1}_{q,1}}
\|\theta_1-\theta_2\|_{B^{s+1}_{q,1}} + \|\theta_2\|_{B^{s+1}_{q,1}}\|\dv(\bu_1-\bu_2)\|_{B^{s+1}_{q,1}})\\
&\quad \leq C(\|\bu_1\|_{B^{s+2}_{q,1}}\|\pd_t(\theta_1-\theta_2)\|_{L_1((0,T), B^{s+1}_{q,1})}
+ \|\pd_t\theta_2\|_{L_1((0, T), B^{s+1}_{q,1})}\|\bu_1-\bu_2\|_{B^{s+2}_{q,1}})
\end{align*}
Using $E_T(\theta_i, \bu_i) \leq \omega$ ($i=1,2$), we have 
$$
\|\theta_1\dv\bu_1-\theta_2\dv\bu_2\|_{L_1((0, T), B^{s+1}_{q,1}) }
\leq C\omega E_T(\theta_1-\theta_2, \bu_1-\bu_2).
$$ 
Write 
\begin{align*}
F(\theta_1+\theta_0, \bu_1) - F(\theta_2+\theta_0, \bu_2)
&= (\theta_1-\theta_2)((\BI - \BA_{\bu_1}):\nabla\bu_1 \\
& -(\theta_0+\theta_2)(\BA_{\bu_1} - \BA_{\bu_2}):\nabla \bu_1
+ (\theta_0+\theta_2)(\BI-\BA_{\bu_2}):\nabla(\bu_1-\bu_2).
\end{align*}
Set $\BI-\BA_{\bu} = F(\int^t_0\nabla \bu)$ and 
write
\begin{align*}
&F(\int^t_0\nabla \bu_1\,\d\ell) - F(\int^t_0\nabla\bu_2\,\d\ell) 
= \int^1_0F'(\int^1_0 (\nabla\bu_2+ m\nabla(\bu_1-\bu_2))\,\d\ell)\,\d m
\int^t_0\nabla(\bu_1-\bu_2)\,\d\ell  \\
& = \Bigl\{F'(0) + \int^1_0\int^1_0F''(n\int^1_0 (\nabla\bu_2+ m\nabla(\bu_1-\bu_2))\,\d\ell)\,\d m \d n\Bigr\}
\int^t_0\nabla(\bu_1-\bu_2)\,\d\ell.
\end{align*}
By \eqref{defin:2}, we have 
\begin{align*}
&\sup_{t \in (0, T)}\Bigl\|n\int^1_0 (\nabla\bu_2+ m\nabla(\bu_1-\bu_2))\,\d\tau\Bigr\|_{L_\infty}\\
&\leq  (1-m)\sup_{t \in (0, T)}\Bigl\|\int^1_0 \nabla\bu_2\,\d\tau\Bigr\|_{L_\infty}
+ m\sup_{t \in (0, T)} \Bigl\|\int^1_0 \nabla\bu_1\,\d\tau\Bigr\|_{L_\infty}
\leq Cc_1 
\end{align*}
by Lemmas \ref{lem:APH} and \ref{lem:Hasp},  we have
$$\|F(\int^t_0\nabla \bu_1\,\d\ell) - F(\int^t_0\nabla\bu_2\,\d\ell) \|_{B^{s+1}_{q,1}}
\leq C\|\bu_1-\bu_2\|_{L_1((0, T), B^{s+2}_{q,1})}.
$$
Thus, by Lemma \ref{lem:APH} , \eqref{theta:1} and \eqref{nonfun:1}, we have
\begin{align*}
&\|F(\theta_1+\theta_0, \bu_1) - F(\theta_2+\theta_0, \bu_2)\|_{B^{s+1}_{q,1}}\\
&\quad \leq C\{\|\theta_1-\theta_2\|_{B^{s+1}_{q,1}}\|\BI - \BA_{\bu_1}\|_{B^{s+1}_{q,1}}\|\nabla\bu_1\|_{B^{s+1}_{q,1}}
+\|\theta_0+\theta_2\|_{B^{s+1}_{q,1}}\|\BA_{\bu_1} - \BA_{\bu_2}\|_{B^{s+1}_{q,1}}\|\nabla \bu_1\|_{B^{s+1}_{q,1}}\\
&\quad + \|\theta_0+\theta_2\|_{B^{s+1}_{q,1}}\|\BI-\BA_{\bu_2}\|_{B^{s+1}_{q,1}}
\|\nabla(\bu_1-\bu_2)\|_{B^{s+1}_{q,1}}\}\\
&\quad\leq C(\|\pd_t(\theta_1-\theta_2)\|_{L_1((0, T), B^{s+1}_{q,1})}\|\bu_1\|_{L_1((0, T), B^{s+1}_{q,1})}
\|\nabla\bu_1\|_{B^{s+1}_{q,1}}\\
&\quad + (\|\theta_0\|_{B^{s+1}_{q,1}}+\|\pd_t\theta_2\|_{L_1((0, T), B^{s+1}_{q,1})})
\|\bu_1-\bu_2\|_{L_1((0, T), B^{s+2}_{q,1}}
\|\nabla \bu_1\|_{B^{s+1}_{q,1}}
\\
&\quad +(\|\theta_0\|_{B^{s+1}_{q,1}}+\|\pd_t\theta_2\|_{L_1((0, T), B^{s+1}_{q,1})})
\|\nabla\bu_2\|_{L_1((0, T), B^{s+1}_{q,1})}\|\nabla(\bu_1-\bu_2)\|_{B^{s+1}_{q,1}}).
\end{align*}
Using the conditions: $E_T(\theta_i, \bu_i) \leq \omega$ ($i=1,2$), we have
$$\|F(\theta_1+\theta_0, \bu_1) - F(\theta_2+\theta_0, \bu_2)\|_{L_1((0, T), B^{s+1}_{q,1})}
\leq C(\omega+\omega^2)E_T(\theta_1-\theta_2, \bu_1-\bu_2),
$$
where $C$ depends on $\|\eta_0\|_{B^{s+1}_{q,1}}$. In fact, we estimate 
$$\|\theta_0\|_{B^{s+1}_{q,1}}+\|\pd_t\theta_2\|_{L_1((0, T), B^{s+1}_{q,1})}
\leq \|\theta_0-\eta_0\|_{B^{s+1}_{q,1}} + \|\eta_0\|_{B^{s+1}_{q,1}} + 
\|\pd_t\theta_2\|_{L_1((0, T), B^{s+1}_{q,1})}
\leq 2\omega + \|\eta_0\|_{B^{s+1}_{q,1}}.
$$
Summing up, we have obtained
\begin{equation}\label{diff:4}
\|\BF\|_{L_1((0, T), B^{s+1}_{q,1})} 
\leq C(\omega+ \omega^2)E_T(\theta_1-\theta_2, \bu_1-\bu_2),
\end{equation}
for some constant $C$ depending on $\|\eta_0\|_{B^{s+1}_{q,1}}$. \par 

Now, we treat $\BG$.  First, we estimate $\tilde\bG(\theta_1, \bu_1) - \tilde\bG(\theta_2, \bu_2)$.
Write
\begin{align*}
&\tilde\bG(\theta_1, \bu_1) - \tilde\bG(\theta_2, \bu_2) 
 = (\eta^\epsilon_0-\theta_0)\pd_t(\bu_1-\bu_2) \\
&+ \nabla (P(\theta_0+\theta_1) - P(\theta_0) - P'(\eta^\epsilon_0)\theta_1
-(P(\theta_0+\theta_2) - P(\theta_0) - P'(\eta^\epsilon_0)\theta_2))\\
& = (\eta^\epsilon_0-\theta_0)\pd_t(\bu_1-\bu_2) +
\nabla \int^1_0P''(\eta^\epsilon_0 + \ell(\theta_0 - \eta^\epsilon_0))\,\d\ell
(\theta_0-\eta^\epsilon_0)(\theta_1-\theta_2) \\
&+\nabla\{\int^1_0(1-\ell)(P''(\theta_0+\ell\theta_1) - P''(\theta_0+\ell\theta_2))\,\d\ell \theta_1^2
+ \int^1_0(1-\ell)(P''(\theta_0+\ell\theta_2)\,\d\ell (\theta_1^2-\theta_2^2) \}.
\end{align*}
Writing $\eta^\epsilon_0 + \ell(\theta_0-\eta^\epsilon_0) = \eta_0 + (1-\ell)(\eta^\epsilon_0-\eta_0)
+ \ell(\theta_0-\eta_0)$, using \eqref{domain:4.1} and \eqref{assump:0}, we may assume that 
$$\rho_1/2-\gamma_* < \tilde\eta^\epsilon_0 + \ell(\theta_0-\eta_0^\epsilon)
< 2\rho_2-\gamma_*$$
for any $\ell \in (0, 1)$, and so we write
\begin{align*}
\int^1_0P''(\eta^\epsilon_0 + \ell(\theta_0 - \eta^\epsilon_0))\,\d\ell
= \int^1_0\int^1_0 P'''(\gamma_* + m(\tilde\eta^\epsilon_0 + 
 \ell(\theta_0 - \eta^\epsilon_0))(\tilde\eta^\epsilon_0 + 
 \ell(\theta_0 - \eta^\epsilon_0))\,\d\ell\d m + P''(\gamma_*).
\end{align*}
Thus,  by Lemmas \ref{lem:APH} and \ref{lem:Hasp} and \eqref{smalldist:1}, we have
\begin{align*}
&\|\nabla( \int^1_0P''(\eta^\epsilon_0 + \ell(\theta_0 - \eta^\epsilon_0))\,\d\ell
(\theta_0-\eta^\epsilon_0)(\theta_1-\theta_2))\|_{B^s_{q,1}}
\leq C(\gamma_*, \|\tilde\eta_0\|_{B^{s+1}_{q,1}})\omega\|\theta_1-\theta_2\|_{B^{s+1}_{q,1}}.
\end{align*}
Write
$$\theta_0 + \ell\theta_2 + m(\theta_0 + \ell\theta_1-(\theta_0+\ell\theta_2))
= \eta_0 + (\theta_0-\eta_0) + \ell\theta_2 + m\ell(\theta_1-\theta_2).
$$
Since 
\begin{align*}
\|(\theta_0-\eta_0) + \ell\theta_2 + m\ell(\theta_1-\theta_2)\|_{L_\infty}
&\leq C(\|\theta_0-\eta_0\|_{B^{s+1}_{q,1}}
+\ell(1-m)\|\theta_2\|_{B^{s+1}_{q,1}} + \ell m\|\theta_1\|_{B^{s+1}_{q,1}})\\
&\leq C(\omega+\sum_{i=1,2}\|\pd_t\theta_i\|_{L_1((0, T), B^{s+1}_{q,1})})
\leq C\omega,
\end{align*}
we may assume that 
$$\rho_1/2-\gamma_* < \tilde\eta_0 + (\theta_0-\eta_0)+\ell\theta_2
+ m\ell(\theta_1-\theta_2) < 2\rho_2-\gamma_*,$$
and so, we write
\begin{align*}
&\int^1_0(1-\ell)(P''(\theta_0+\ell\theta_1)-P''(\theta_0 + \ell\theta_2))\,\d\ell\theta_1^2\\
&= \int^1_0\int^1_0(1-\ell)P'''(\theta_0+\ell\theta_2+m\ell(\theta_1-\theta_2))
(\theta_1-\theta_2)\,\d\ell\,\d m\, \theta_1^2\\
& = \int^1_0\int^1_0\int^1_0(1-\ell)P''''(\gamma_*+n(\tilde\eta_0+(\theta_0-\eta_0)+\ell\theta_2
+m\ell(\theta_1-\theta_2))) \\
&\qquad \times(\tilde\eta_0+(\theta_0-\eta_0)+\ell\theta_2+m\ell(\theta_1-\theta_2))
\,\d\ell\d m\d n\,(\theta_1-\theta_2)\theta_1^2
+ \frac12P'''(\gamma_*)(\theta_1-\theta_2)\theta_1^2.
\end{align*}
By Lemmas \ref{lem:APH} and \ref{lem:Hasp}, and \eqref{theta:1}, 
we have
\begin{align*}
&\|\nabla(\int^1_0(1-\ell)(P''(\theta_0+\ell\theta_1)-P''(\theta_0 + \ell\theta_2))\,\d\ell\theta_1^2)
\|_{B^s_{q,1}}\\
&\quad \leq C(\gamma_*, \|\tilde\eta_0\|_{B^{s+1}_{q,1}})\|\pd_t\theta_1\|_{L_1((0, T), B^{s+1}_{q,1}}^2
\|\theta_1-\theta_2\|_{B^{s+1}_{q,1}}.
\end{align*}
Concerning the last term, we write $\theta_0 + \ell\theta_2 = \eta_0 + (\theta_0-\eta_0) + \ell\theta_2$.  Since
$$\|\theta_0-\eta_0 + \ell\theta_2\|_{L_\infty} \leq C(\|\theta_0-\eta_0\|_{B^{s+1}_{q,1}}+
\|\theta_2\|_{B^{s+1}_{q,1}})
\leq C(\|\theta_0-\eta_0\|_{B^{s+1}_{q,1}} + \|\pd_t\theta_2\|_{L_1((0, T), B^{s+1}_{q,1})})
\leq C\omega,$$
choosing $\omega>0$ small enough, we may assume that 
$$\rho_1/2-\gamma_* < \tilde\eta_0 + (\theta_0-\eta_0) + \ell\theta_2 < 2\rho_2-\gamma_*$$
for any $\ell \in (0, 1)$.  Thus, writing
\begin{align*}
&\int^1_0(1-\ell)P''(\theta_0 + \ell\theta_2)\,\d\ell\,(\theta_1^2-\theta_2^2)\\
&= \Bigl\{\frac12P''(\gamma_*)
+ \int^1_0\int^1_0(1-\ell)P'''(\gamma_*+m(\tilde\eta_0 + \theta_0-\eta_0+\ell\theta_2))
(\tilde\eta_0 + \theta_0-\eta_0+\ell\theta_2)\,\d\ell\d m\}\\
&\qquad\times(\theta_1-\theta_2)(\theta_1+\theta_2),
\end{align*}
By Lemmas \ref{lem:APH} and \ref{lem:Hasp}, and \eqref{theta:1}, we have
\begin{align*}
&\|\nabla(\int^1_0(1-\ell)P''(\theta_0 + \ell\theta_2)\,\d\ell\,(\theta_1^2-\theta_2^2))
\|_{B^s_{q,1}} \\
&\quad \leq C(\gamma_*, \|\tilde\eta_0\|_{B^{s+1}_{q,1}})
(\|\tilde\eta^\epsilon_0\|_{B^{s+1}_{q,1}}+\|\theta_0-\eta_0\|_{B^{s+1}_{q,1}}
+\|\pd_t\theta_2\|_{L_1((0, T),B^{s+1}_{q,1})})\\
&\qquad\times (\|\pd_t\theta_1\|_{L_1((0, T),B^{s+1}_{q,1})}
+ \|\pd_t\theta_2\|_{L_1((0, T),B^{s+1}_{q,1})})\|\theta_1-\theta_2\|_{B^{s+1}_{q,1}}.
\end{align*}
Summing up, we have obtained 
\begin{align*}
&\|\tilde\bG(\theta_1, \bu_1) - \tilde\bG(\theta_2, \bu_2) \|_{L_1((0, T), B^s_{q,1}} 
\leq C\|\eta^\epsilon_0-\theta_0\|_{B^{N/q}_{q,1}}\|\pd_t(\bu_1-\bu_2)\|_{L_1((0, T), B^s_{q,1}} \\
&\quad  + C(\gamma_*, \|\eta_0\|_{B^{s+1}_{q,1}})(\omega 
+ \|\pd_t\theta_1\|_{L_1((0, T), B^{s+1}_{q,1})}^2
+ \sum_{i=1}^2\|\pd_t\theta_i\|_{L_1((0, T), B^{s+1}_{q,1})})\|\theta_1-\theta_2\|_{L_1((0, T), B^{s+1}_{q,1})}.
\end{align*}
Since $E_T(\theta_i, \bu_i) \leq \omega$, using \eqref{semi:1}, we have
$$\|\tilde\bG(\theta_1, \bu_1) - \tilde\bG(\theta_2, \bu_2) \|_{B^s_{q,1}} 
\leq C(\omega+\omega^2)E_T(\theta_1-\theta_2, \bu_1-\bu_2).
$$
Finally, we estimate $\bG(\theta_0+\theta_1, \bu_1)-\bG(\theta_0+\theta_2, \bu_2)$. 
We write
\begin{align*}
&\bG(\theta_0+\theta_1, \bu_1)-\bG(\theta_0+\theta_2, \bu_2)\\
&= ((\BA_{\bu_2}^\top)^{-1}-(\BA_{\bu_1}^\top)^{-1})(\theta_0+\theta_1)\pd_t\bu_1
+ (\BI-(\BA_{\bu_2}^\top)^{-1})(\theta_1-\theta_2)\pd_t\bu_1
+ (\BI-(\BA_{\bu_2}^\top)^{-1})(\theta_0+\theta_2) \pd_t(\bu_1-\bu_2)\\
& + \alpha ((\BA_{\bu_1}^\top)^{-1}-(\BA_{\bu_2}^\top)^{-1})\dv(\BA_{\bu_1}\BA_{\bu_1}^\top:\nabla\bu_1)
+ \alpha((\BA_{\bu_2}^\top)^{-1}-\BI)\dv((\BA_{\bu_1}\BA_{\bu_1}^\top-\BA_{\bu_2}\BA_{\bu_2}^\top):\nabla\bu_1)\\
&+  \alpha((\BA_{\bu_2}^\top)^{-1}-\BI)\dv(\BA_{\bu_2}\BA_{\bu_2}^\top:\nabla(\bu_1-\bu_2))
+ \alpha\dv((\BA_{\bu_1}-\BA_{\bu_2})(\BA_{\bu_1}^\top-\BI):\nabla\bu_1) \\
& + \alpha\dv( \BA_{\bu_2}(\BA_{\bu_1}^\top-\BA_{\bu_2}^\top):\nabla\bu_1) 
+ \alpha\dv (\BA_{\bu_2}(\BA_{\bu_2}^\top-\BI):\nabla(\bu_1-\bu_2) )\\
&+ \beta\nabla((\BA_{\bu_1}^\top - \BA_{\bu_2}^\top):\nabla\bu_1)
+ \beta\nabla((\BA_{\bu_2}^\top-\BI):\nabla(\bu_1-\bu_2)).
\end{align*}
Employing the similar argument to the proof of \eqref{diff:4}, we have
\begin{align*}
\|(\BA_{\bu_1}^\top)^{-1} -(\BA_{\bu_2}^\top)^{-1}\|_{B^{N/q}_{q,1}} 
&\leq C\|\nabla(\bu_1-\bu_2)\|_{L_1((0, T), B^{N/q}_{q,1})}, \quad
\|(\BI-(\BA_{\bu_i})^\top\|_{B^{N/q}_{q,1}} \leq C\|\nabla \bu_i\|_{L_1((0, T), B^{N/q}_{q,1})}, \\
\|\BA_{\bu_1}^\top -\BA_{\bu_2}^\top\|_{B^{N/q}_{q,1}} 
&\leq C\|\nabla(\bu_1-\bu_2)\|_{L_1((0, T), B^{N/q}_{q,1})}, \quad 
\|\BA_{\bu_i}\BA_{\bu_i}^\top-\BI\|_{B^{s+1}_{q,1}}\leq C\|\nabla\bu_i\|_{L_1((0, T), B^{s+1}_{q,1})}, \\
	\|\BA_{\bu_1}\BA_{\bu_1}^\top - \BA_{\bu_2}\BA_{\bu_2}^\top\|_{B^{N/q}_{q,1}}
&\leq  C\|\nabla(\bu_1-\bu_2)\|_{L_1((0, T), B^{N/q}_{q,1})}.
\end{align*}
Therefore, by Lemmas \ref{lem:APH} and \ref{lem:Hasp}, we have
\begin{align*}
&\|\bG(\theta_0+\theta_1, \bu_1)-\bG(\theta_0+\theta_2, \bu_2)\|_{B^s_{q,1}} \\
& \leq C\{\|(\BA_{\bu_2}^\top)^{-1}-(\BA_{\bu_1}^\top)^{-1}\|_{B^{N/q}_{q,1}}
(\|\theta_0\|_{B^{N/q}_{q,1}} + \|\pd_t\theta\|_{L_1((0, T), B^{N/q}_{q,1})})
\|\pd_t\bu_1\|_{B^s_{q,1}} \\
&+ \|\nabla\bu_2\|_{L_1((0, T), B^{N/q}_{q,1})}\|\pd_t(\theta_1-\theta_2)\|_{L_1((0, T), 
B^{N/q}_{q,1})}\|\pd_t\bu_1\|_{B^s_{q,1}} \\
&  + \|\nabla\bu_2\|_{L_1((0, T), B^{N/q}_{q,1})}(\|\theta_0\|_{B^{N/q}_{q,1}}+
\|\pd_t\theta_2\|_{L_1((0, T), B^{N/q}_{q,1})})\|\pd_t(\bu_1-\bu_2)\|_{B^s_{q,1}} \\
&+ \|(\BA_{\bu_2}^\top)^{-1}-(\BA_{\bu_1}^\top)^{-1}\|_{B^{N/q}_{q,1}}
(1+\|\nabla\bu_1\|_{B^{s+1}_{q,1}})\|\bu_1\|_{B^{s+1}_{q,1}} \\
&+ \|\nabla\bu_2\|_{L_1((0, T), B^{N/q}_{q,1})}\|\BA_{\bu_1}\BA_{\bu_1}^\top
-\BA_{\bu_2}\BA_{\bu_2}^\top\|_{B^{s+1}_{q,1}}\|\nabla\bu_1\|_{B^{s+1}_{q,1}}
\\
&+  \|\nabla\bu_2\|_{L_1((0, T), B^{N/q}_{q,1})}
(1+ \|\nabla\bu_2\|_{L_1((0, T), B^{N/q}_{q,1})})\|\nabla(\bu_1-\bu_2) \|_{B^{s+1}_{q,1}} \\
&+ \|\nabla(\bu_1-\bu_2)\|_{L_1((0, T), B^{s+1}_{q,1})}\|\nabla\bu_1\|_{L_1((0, T), B^{s+1}_{q,1})}
\|\nabla\bu_1\|_{B^{s+1}_{q,1}} \\
& + (1+\|\nabla\bu_2\|_{L_1((0, T), B^{s+1}_{q,1})})\|\nabla(\bu_1-\bu_2)\|_{L_1((0, T), B^{s+1}_{q,1})}
\|\nabla\bu_1\|_{B^{s+1}_{q,1}} \\
&+ (1+\|\nabla\bu_2\|_{L_1((0, T), B^{s+1}_{q,1})})\|\nabla\bu_2\|_{L_1((0, T), B^{s+1}_{q,1})}
\|\nabla(\bu_1-\bu_2)\|_{B^{s+1}_{q,1}} \\
&+ \|\nabla(\bu_1 -\bu_2)\|_{L_1((0, T), B^{s+1}_{q,1})}\|\nabla\bu_1\|_{B^{s+1}_{q,1}} 
+ \|\nabla\bu_2\|_{L_1((0, T), B^{s+1}_{q,1})}\|\nabla(\bu_1-\bu_2)\|_{B^{s+1}_{q,1}}\}.
\end{align*}
We have $\|\theta_0\|_{B^{s+1}_{q,1}} \leq \|\theta_0-\eta_0\|_{B^{s+1}_{q,1}} + \|\tilde\eta_0\|_{B^{s+1}_{q,1}}
\leq C\omega + \|\tilde\eta_0\|_{B^{s+1}_{q,1}}$. 
Thus, we have
$$\|\bG(\theta_0+\theta_1, \bu_1)-\bG(\theta_0+\theta_2, \bu_2)\|_{L_1(0, T), B^s_{q,1})} 
\leq C(\gamma_*, \|\tilde\eta_0\|_{B^{s+1}_{q,1}})(\omega+\omega^2)E_T(\theta_1-\theta_2,
\bu_1-\bu_2).$$
Summing up, we have obtained \eqref{diff:3}. \par
Combining \eqref{diff:2} and \eqref{diff:3} yields 
$$E_T(\eta_1-\eta_2, \bw_1-\bw_2) \leq Ce^{\gamma T}(\omega+\omega^2)
E_T(\theta_1-\theta_2, \bu_1,\bu_2).$$
Thus, first we choose $\omega>0$ so small that $Ce(\omega+\omega^2) \leq 1/2$, 
and second we choose  $T>0$ so small that $\gamma T \leq 1$, we have 
$$E_T(\eta_1-\eta_2, \bw_1-\bw_2) \leq (1/2)
E_T(\theta_1-\theta_2, \bu_1,\bu_2),
$$
which shows that 
$\Phi$ is a contraction map from $S_{T, \omega}$ into itself. Therefore, by the Banach fixed
point theorem, $\Phi$ has a unique fixed point $(\eta, \bw) \in S_{T, \omega}$.
In \eqref{st:2}, setting $(\eta, \bw) = (\theta, \bu)$ and recalling 
$\rho = \theta_0+\theta$ and $\tilde\bG(\theta, \bu) = (\eta^\epsilon_0-\theta_0-\theta)\pd_t\bu
- \nabla(P(\theta_0+\theta)-P(\theta_0)-P'(\eta^\epsilon_0)\theta)$, we see that $\theta$ and $\bu$ satisfy
equations:
\begin{equation}\label{st:4}\left\{\begin{aligned}
\pd_t\theta+\eta^\epsilon_0\dv\bu = (\eta^\epsilon_0-\theta_0-\theta)\dv\bu +  F(\theta_0+\theta, \bu)& 
&\quad&\text{in $\HS\times(0, T)$}, \\
\eta^\epsilon_0\pd_t\bu - \alpha\Delta \bu  -\beta\nabla\dv\bu
+  \nabla(P'(\eta^\epsilon_0) \theta) = -\nabla P(\theta_0) + \bG(\theta_0+\theta, \bu)
- \tilde\bG(\theta, \bu)
& &\quad&\text{in $\HS\times(0, T)$}, \\
\bu|_{\pd\HS} =0, \quad (\eta, \bu)|_{t=0} = (0, \bu_0)
& &\quad&\text{in $\HS$}.
\end{aligned}\right.\end{equation}
Thus, setting $\rho = \theta_0 + \theta$, from \eqref{st:4} it follows that 
$\rho$ and $\bu$ satisfy equations \eqref{ns:2}.  Moreover, 
$(\rho, \bu)$ belongs to $S_{T,\omega}$, which completes the proof of 
Theorem \ref{thm:2}.

{\bf A proof of Theorem \ref{thm:1}}~ As was mentioned at the beginning of 
Subsec. \ref{sec.1.1}, $y= X_\bu(x, t)$ is a $C^1$ diffeomorphism from $\Omega$ onto itself
for any $t \in (0, T)$, because $\bu \in L_1((0, T), B^{s+2}_{q,1}(\Omega)^N)$.
Let $x = X_\bu^{-1}(y, t)$ be the inverse of $X_\bu$. For any function 
$F \in B^s_{q,1}(\HS)$, $1 < q < \infty$, $s \in \BR$, it follow from the chain rule that 
$$\|F\circ X_\bu^{-1}\|_{B^s_{q,1}(\HS)} \leq C\|F\|_{B^s_{q,1}(\HS)}$$
with some constant $C>0$ (cf. Amann \cite[Theorem 2.1]{Amann00}). 
Let $(\rho, \bv) = (\theta, \bu)\circ X_\bu^{-1}$ and $\BA_\bu = (\nabla_yX_\bu)^{-1}$.
Let $\BA_\bu^\top = (A_{jk})$.  Since there holds
\begin{gather*}
\nabla_y(\rho, \bv) = (\BA_\bu^\top \nabla_x(\theta, \bu))\circ X_\bu^{-1}, \\
\pd_{y_j}\pd_{y_k}\bv = \sum_{\ell, \ell'}
A_{j\ell}\pd_{y_\ell}(A_{k\ell'}\pd_{y_{\ell'}}\bu))\circ X_\bu^{-1}
\quad(j, k =1, \ldots, N).
\end{gather*}
Concerning the time derivative of $\rho$ and $\bv$, we rely on the relation:
$$\pd_t(\rho, \bv) = \pd_t(\theta, \bu)\circ X_\bu^{-1}
-((\bu\circ X_\bu^{-1})\cdot\nabla_y)(\rho, \bv).
$$
Therefore, by Theorem \ref{thm:2} and Lemma \ref{lem:APH}, 
we arrive at \eqref{main.reg}.  This completes the proof of Theorem \ref{thm:1}. 


\end{document}